\documentclass[reqno,12pt]{amsart}

\usepackage{tikz}

\usetikzlibrary{babel}
\usetikzlibrary{calc}
\usepackage{tkz-euclide}
\usepackage{amsthm}
\usepackage{amsmath}
\usepackage{amssymb}
\usepackage{latexsym}
\usepackage{graphicx}
\usepackage{url}
\usepackage{pdfpages}
\usepackage{ulem} 
\usepackage{enumitem}
\usepackage{xcolor}
\usepackage{mathtools}
\mathtoolsset{showonlyrefs}

\usepackage{ulem}
\usepackage[colorlinks, linkcolor=blue, filecolor=blue, citecolor=olive, urlcolor=purple]{hyperref}

\usepackage[usenames,dvipsnames]{pstricks}
\usepackage{pst-grad}
\usepackage{pst-plot}

\def\blue{\color{blue}}

\definecolor{violet(ryb)}{rgb}{0.53, 0.0, 0.69}
\definecolor{internationalorange}{rgb}{1.0, 0.31, 0.0}

\newtheorem{theorem}{Theorem}[section]
\newtheorem{lemma}[theorem]{Lemma} 
\newtheorem{definition}[theorem]{Definition}
\newtheorem{proposition}[theorem]{Proposition}

\newtheorem{remark}[theorem]{Remark}

\newtheorem*{theorem*}{\it Theorem} 

\def\a{{\bf a}}

\def\R{\mathbb R} 
\def\N{\mathbb N} 

\newcommand{\norm}[1]{\left\Vert #1 \right\Vert} 

\def\vint_#1{\mathchoice%
          {\mathop{\kern 0.2em\vrule width 0.6em height 0.69678ex depth -0.58065ex
                  \kern -0.8em \intop}\nolimits_{\kern -0.4em#1}}%
          {\mathop{\kern 0.1em\vrule width 0.5em height 0.69678ex depth -0.60387ex
                  \kern -0.6em \intop}\nolimits_{#1}}%
          {\mathop{\kern 0.1em\vrule width 0.5em height 0.69678ex
              depth -0.60387ex
                  \kern -0.6em \intop}\nolimits_{#1}}%
          {\mathop{\kern 0.1em\vrule width 0.5em height 0.69678ex depth -0.60387ex
                  \kern -0.6em \intop}\nolimits_{#1}}}
\def\vintslides_#1{\mathchoice%
          {\mathop{\kern 0.1em\vrule width 0.5em height 0.697ex depth -0.581ex
                  \kern -0.6em \intop}\nolimits_{\kern -0.4em#1}}%
          {\mathop{\kern 0.1em\vrule width 0.3em height 0.697ex depth -0.604ex
                  \kern -0.4em \intop}\nolimits_{#1}}%
          {\mathop{\kern 0.1em\vrule width 0.3em height 0.697ex depth -0.604ex
                  \kern -0.4em \intop}\nolimits_{#1}}%
          {\mathop{\kern 0.1em\vrule width 0.3em height 0.697ex depth -0.604ex
                  \kern -0.4em \intop}\nolimits_{#1}}}

\renewcommand{\d}{\,\mathrm{d}} 
\def\1{\raisebox{2pt}{\rm{$\chi$}}} 

\def\e{\mathsf{e}} 	
\def\v{\mathsf{v}}	
\newcommand{\vi}{\mathsf{i}} 
\newcommand{\vf}{\mathsf{f}} 

\newcommand{\G}{\mathsf{G}}	

\newcommand{\V}{\mathsf{V}}	
\newcommand{\E}{\mathsf{E}}	

\numberwithin{equation}{section} 

\usepackage{mathtools}
\mathtoolsset{showonlyrefs}

\begin{document}

\title[Nonlinear diffusion on metric graphs]{\bf Doubly Nonlinear Diffusion  Equations  on Metric Graphs}

\author[J. M. Maz\'on and J. Toledo ]{José M. Maz\'on and Julián Toledo}

 \address{ J. M. Maz\'{o}n: Departamento de An\'{a}lisis Matem\'{a}tico,
Univ. Valencia, Av. Vicent Andrés Estellés,  19, 46100 Burjassot, Spain.
 {\tt mazon@uv.es}}
\address{J. Toledo: Departamento de An\'{a}lisis Matem\'{a}tico,
Univ. València, Av. Vicent Andrés Estellés,  19, 46100 Burjassot, Spain.
 {\tt toledojj@uv.es}}



\thanks{
}

\setcounter{tocdepth}{1}

\date{\bf\blue\today}

\begin{abstract}
In this paper we study existence and uniqueness of solutions for  a very general class of doubly nonlinear diffusion equations  on metric graphs,  which provide the appropriate mathematical framework to describe complex tubular networks in which axial diffusion is the main focus.  Some important particular cases  covered in our study are the Porous Medium Equation and the evolution equation for the $p$-Laplacian,  but we also consider the case in that diffusion changes from one edge to another,  which takes into account the influence of the properties of the tubules forming the network on axial diffusion. Furthermore, the problem is studied under  non-homogeneous  Neumann-Kirchhoff conditions on the vertices of the graph.

\end{abstract}

\maketitle


\section{Introduction}

Many physical, biological and engineering  dynamical systems are posed on  thin  elongated structures arranged in complex networks. Examples include quantum wires and nanostructures in mesoscopic physics, waveguides and optical fibers, vascular and renal tubular networks, as well as microfluidic devices.  A common feature of such systems is that their transversal dimensions are much smaller than their longitudinal extent,   while the underlying connectivity and topology play a dominant role in the observed dynamics.

From a mathematical perspective, these  dynamical systems are naturally modeled as partial differential equations posed on thin tubular domains in $\mathbb{R}^n$.   Typical examples involve diffusion, wave propagation, or Schrödinger-type equations.   However, direct analysis of such problems in their full dimensional setting is often analytically intractable and computationally expensive. This motivates the search for effective reduced models that retain the essential physical and geometric features while significantly simplifying the underlying equations.

Quantum graphs arise precisely as such effective models. They consist of metric graphs (graphs whose edges are identified with finite or infinite intervals) equipped with differential operators acting along the edges and suitable coupling conditions at the vertices. The term {\it quantum} originates from early applications to quantum mechanics, where Schrödinger operators on networks of thin wires were studied. Nevertheless, the framework is by no means restricted to quantum systems and has since found applications in diffusion processes, wave propagation, elasticity, and biological transport phenomena.  Indeed, the literature on quantum graphs is vast and extensive, and there is no chance to give even a brief overview of the subject here.
We only mention a few recent monographs and collected works with a comprehensive bibliography \cite{BerCarFul06}, \cite{BerKuc13}, \cite{ExnKeaKuc08}, \cite{KosSch06}, \cite{Mug14} and \cite{Pos09b}.  See also~\cite{Scotel} and the references therein.

In the Euclidean space $\R^N$, the doubly nonlinear diffusion equation for the $p$-Laplacian,
\begin{equation}\label{DNL1}
\frac{\partial v}{ \partial t} -\Delta_p u=f, \quad v \in \gamma(u),
\end{equation}  with boundary conditions and initial data,
where $\gamma$ is a maximal monotone graph in $\R \times \R$, $f(t,x)$ is a source, and
$$\Delta_p u:= {\rm div}( \vert \nabla u \vert^{p-2} \nabla u), \quad 1 < p < \infty,$$ appears in different fields, such as flow in porous media, plasma physics or  image processing. This type of equations was proposed by Leibenzon in modeling
turbulent fluid problems in 1947 (\cite{L}), and  has been investigated mathematically by different authors,  we should mention the earlier work  \cite{Ba}  and the Kalashnikov’s well-known survey paper \cite{K}.   The prototype of equation \eqref{DNL1} is given by
\begin{equation}\label{DNLProt}
\frac{\partial u^m}{ \partial t} -\Delta_p u=f
\end{equation}
(read $u^m=|u|^{m-1}u$).  A quite exhaustive study of this type of problems is done in~\cite{AndIgMazTo06} and \cite{AndIgMazTo07} (see also \cite{BDMS18}).
For $m=1$ and $p =2$, equation  \eqref{DNLProt} reduces to
the standard heat equation. For $m = 1$ and $p \in  (1, \infty)$ the equation is known as the $p$-Laplacian   evolution  equation, while for $m \in (0, \infty)$ and $p=2$   it is  the porous medium equation  (see the monograph~\cite{V}).   
The case $p - 1 > m$ is
commonly known as a {\it slow diffusion equation}, while the case $p -1 > m$ is named the {\it fast diffusion equation}. 
The case $p=1$ has been recently study in \cite{MMT23}.

 Our aim in this article is to study some doubly nonlinear diffusion problems in metric graphs, which constitute the appropriate mathematical spaces for describing complex tubular networks where axial diffusion is the issue of interest. The axial diffusion process will undoubtedly depend on different factors of the tubules that make up the structure, leading to different possible operators to drive diffusion in each of them. We will assume that the diffusion equations are of the type illustrated in~\eqref{DNLProt}, but changing from one edge to another to account for the above observation. We will also address zero or non-zero flow conditions at the vertices and allow a source around the edges.
Concretely the problem we want to study is 
\begin{equation}\label{evpbintro1711} \left\{\begin{array}{l}
\frac{\partial [v]_\e}{\partial t}-\Delta_{p_\e} [u]_\e= [f]_\e\quad\hbox{in } (0,+\infty)\times(0,\ell_\e),\ \forall\, \e\in \E(\G),\\
\\ {}
 [v]_\e=\gamma_\e([u]_\e)\quad\hbox{ in } (0,\infty)\times (0,\ell_\e),\ \forall\, \e\in \E(\G),
\\ \\
\hbox{ $u$ continuous,}
\\ \\
\displaystyle  \partial^{\overline{p}}_\nu u(t,\v)={\omega(t,\v)}   \quad\hbox{for all } (t,\v)\in (0,+\infty)\times \V(\G),
\\ \\ v(0)=v_0.
\end{array}\right.
\end{equation}
where $\G$ is a metric graph with vertices $\V(\G)$  and edges $\E(\G)$ (see Figure~\ref{fig01}), $(0,\ell_\e)$ is the interval where the edge $\e$ is parametrized (see Subsection~\ref{sec1747}), $-\Delta_{p_\e}$ is the $p_\e$-Laplacian operator (not necessarily the same in each edge),  
 $\gamma_\e:\mathbb{R}\to \mathbb{R}$ is a continuous   increasing  function with $\gamma_\e(\mathbb{R})=\mathbb{R}$ and $\gamma_\e(0)=0$ (not necessarily the same in each edge),  $[f]_\e(t,x)$ is a source term, and $$\partial^{\overline{p}}_\nu u(t,\v)=\omega(t,\v)(=\omega_\v(t))$$   determines the flux trough a vertex $\v$  (see~\eqref{GREENF0op} later on).  
\begin{figure}[ht]
    \centering
    \begin{tikzpicture}[
  grow=right, 
  level 1/.style={sibling distance=4cm, level distance=4cm},
  level 2/.style={sibling distance=2cm, level distance=3cm},
  level 3/.style={sibling distance=1cm, level distance=2cm},
  vertex/.style={draw, circle, minimum size=7mm, inner sep=1pt},
  edge from parent/.style={draw},
  every edge node/.style={draw=none, fill=none, font=\scriptsize}
]

\node[vertex] (v1) {$\v_1$}
  child {node[vertex] (v2) {$\v_2$}
    child {node[vertex] (v4) {$\v_4$}
      child {node[vertex] (v8)  {$\v_8$}}
      child {node[vertex] (v9)  {$\v_9$}}
    }
    child {node[vertex] (v5) {$\v_5$}
      child {node[vertex] (v10) {$\v_{10}$}}
      child {node[vertex] (v11) {$\v_{11}$}}
    }
  }
  child {node[vertex] (v3) {$\v_3$}
    child {node[vertex] (v6) {$\v_6$}
      child {node[vertex] (v12) {$\v_{12}$}}
      child {node[vertex] (v13) {$\v_{13}$}}
    }
    child {node[vertex] (v7) {$\v_7$}
      child {node[vertex] (v14) {$\v_{14}$}}
      child {node[vertex] (v15) {$\v_{15}$}}
    }
  };

\path
  (v1) -- (v2) node[midway, below]  {$\e_1$}
  (v1) -- (v3) node[midway, above]  {$\e_2$}
  (v2) -- (v4) node[midway, below]  {$\e_3$}
  (v2) -- (v5) node[midway, above]  {$\e_4$}
  (v3) -- (v6) node[midway, below]  {$\e_5$}
  (v3) -- (v7) node[midway, above]  {$\e_6$}
  (v4) -- (v8) node[midway, below]  {$\e_7$}
  (v4) -- (v9) node[midway, above]  {$\e_8$}
  (v5) -- (v10) node[midway, below] {$\e_9$}
  (v5) -- (v11) node[midway, above] {$\e_{10}$}
  (v6) -- (v12) node[midway, below] {$\e_{11}$}
  (v6) -- (v13) node[midway, above] {$\e_{12}$}
  (v7) -- (v14) node[midway, below] {$\e_{13}$}
  (v7) -- (v15) node[midway, above] {$\e_{14}$};

\foreach \i in {1,...,15}{
  \draw[dashed, -{Latex[length=2mm]},lightgray] (v\i.east) -- ++(0.6,0)
      node[anchor=west, font=\scriptsize] {\color{gray} $\omega_{\v_{\i}}$};
}
\end{tikzpicture}
    \caption{\small The dashed gray arrows indicate the flux $\omega_{\v_i}$ on each vertex.}
    \label{fig01}
\end{figure}
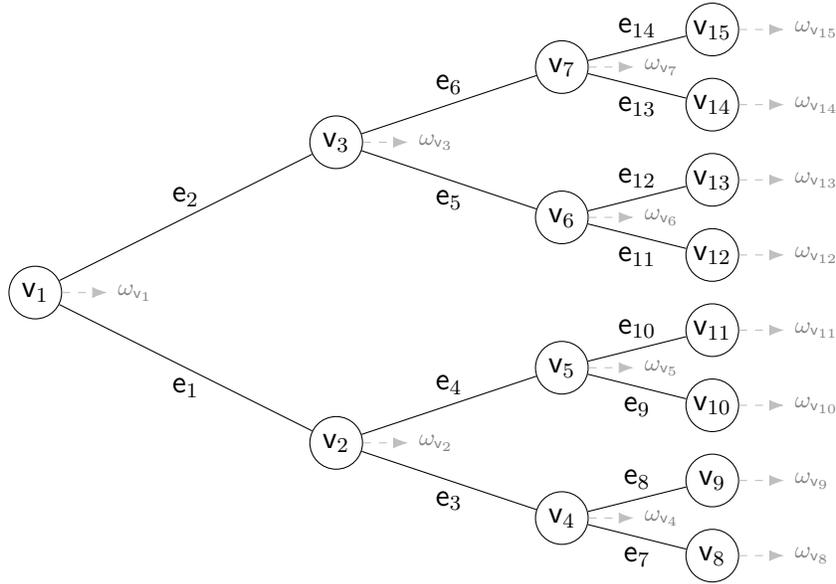

 The pair of conditions 
$$\begin{array}{c}
\hbox{$u$ continuous, and }\\[8pt]  
\partial^{\overline{p}}_\nu u(t,\v)=0,
\end{array}$$ 
is known as the {\it standard Neumann–Kirchhoff condition},  commonly imposed in the literature, among other reasons for its compatibility with operator-theoretic methods. We consider a generalization in which the flux at the vertices may be nonzero and time-dependent.

An important particular case of Problem \eqref{evpbintro1711} is when $p_\e = p$ and $\gamma_\e = \gamma$ for all $\e\in \E(\G)$,    that is, the problem
\begin{equation}\label{evproblem} \left\{\begin{array}{l}
\frac{\partial (\gamma([u]_\e)}{\partial t}-\Delta_{p} [u]_\e=[f]_\e\quad\hbox{in } (0,+\infty)\times(0,\ell_\e),\ \forall\, \e\in \E(\G),\\
\\
\displaystyle  \partial^{p}_\nu u(t,\v)={\omega_\v}   \quad\hbox{for all } (t,\v)\in (0,+\infty)\times \V(\G),
\\ \\ v(0)=v_0,
\end{array}\right.
\end{equation}
In this case, if  $\gamma(r) = r$ for all $r \in \R$, we have a  {\it pure  $p$-Laplacian evolution problem}  in a quantum graph;  and for $p=2$ and  $\gamma(r) = |r|^{m-1}r$, with $m>0$,  we have a {\it pure  porous medium} or {\it fast diffusion equation} quantum graph problem.  
Let us notice that, to our knowledge,   for most of these particular cases, the existence and uniqueness of solutions to  problem~\eqref{evpbintro1711} has not been studied  in the literature.

 To solve the  evolution problem we will handle with three main problems. We  will first solve an elliptic problem associated to the evolution one, which is certainly interesting in itself. We will then prove that we can get a mild solution for the evolution problem using   nonlinear semigroup theory. Finally, we will characterize such solution as the unique weak solution of the evolution problem.
We want to highlight that a principal aspect in this paper is that the proof of the existence of solution   for the associated elliptic problem is done by induction on the number of edges of the graph. It is also worthy to mention  that the main difficulty is to get the Kirchhoff type condition, that is, fixed the flux on a vertex (no matter if it is null or not), to get the continuity   of the temperature, or concentration,~$u$. 
The fact of dealing with  different diffusion rules  on  each edge obeys to describe more realistic practical situations.

\section{Preliminaries}

 \subsection{Metric Graphs}\label{sec1747}

We  recall here some basic knowledge about metric graphs,
see for instance
\cite{BerKuc13} and  the references therein.

A graph $\G$ consists of a finite or countable infinite set
of vertices $\V(\G)=\{\v_i\}$ and a set of edges $\E(\G)=\{\e_j\}$ connecting the
vertices.  A graph $\G$ is said to be a finite graph if the number of edges
and the number of vertices are finite.
 An edge and a vertex on
that edge are called incident. We will denote $\v\in \e$ when the edge $\e$ and the vertex $\v$ are incident.

We define $\E_{\v}(\G)$ as the set of all edges incident to the vertex $\v$, and the {\it degree} of $\v$ as $d_\v:=  \sharp \E_{\v}(\G)$ (the cardinal of $\E_\v(\G)$). We define the {\it boundary} of $\V(\G)$ as
$$\partial \V(\G):= \{ \v \in \V(\G)   :   d_\v =1 \},$$
 and its {\it interior} as
 $$\operatorname{int}(\V(\G)) := \{ \v \in \V(\G)   :   d_\v > 1 \}.$$

  Following~\cite{BerKuc13} we say that a  graph $\G$ is a metric graph if
\begin{enumerate}
	\item each edge $\e$ is assigned  with a positive length $\ell_{\e}\in(0,+\infty];$

\item   for each edge $\e$, a coordinate is assigned to each point of it, including its vertices. For that purpose, each edge $\e$ is  identified with an ordered pair
$(\vi_{\e},\vf_{\e})$ of vertices, being $\vi_{\e}$ and $\vf_{\e}$ the
initial and terminal vertex of $\e$ respectively
and an increasing coordinate in the edge is defined in the direction of the edge:  
	$$
	\begin{array}{rlcc}
 c_\e:&\e&\to& [0,\ell_\e]\\
 &x&\rightsquigarrow& c_\e(x)
 \end{array}
$$
such that $c_\e(\vi_\e)=0$, $c_\e(\vf_\e)=\ell_{\e}$, and it is exhaustive.  This is not restrictive under our standing assumption that there are no multiple edges between any two vertices. Also, an edge $\e$ can be reparametized with the same length in such a way that the initial and final vertices are interchanged by using the parametrization $\tilde c_\e(x)=c_\e(\ell_\e-x)$.
\end{enumerate}

  A graph is said to be
{\it connected} if a path exists between every pair of vertices.

 A {\it compact} metric graph is a finite metric graph whose edges  all have finite
length.

We will deal with   connected and compact metric graphs.  To simplify we also assume that there are no loops (that is there are no edges connecting one vertex), and that there are no two different edges connecting the same vertices.

A function $u$ on a metric graph $\G$ is a collection of functions $[u]_{\e}$
defined on
$(0,\ell_{\e})$  for all $\e\in \E(\G),$
not just at the vertices as in discrete models.

We use the notation (we use the same name for variables in $\G$ and in $(0,\ell_\e)$ since there is not ambiguity):
$$ \int_{\G} u
=
\int_{\G} u(x) \d x
:=
\sum_{\e\in \E(\G)} \int_{0}^{\ell_{\e}} [u]_{\e}(x)\d x=\sum_{\e\in \E(\G)} \int_{0}^{\ell_{\e}} [u]_{\e}.
$$

Let, for each $\e\in \E(\G)$,  $p_\e\in
[1,+\infty]$.   We write $\overline{p}$ to represent  the collection of $(p_\e)_{\e \in \E(\G)}$. We denote by $L^{\overline{p}}(\G)$ the space of all function in $\G$ such that $[u]_\e \in L^{p_\e}(0,\ell_{\e})$  for all $\e\in \E(\G)$, with the norm
$$\Vert u \Vert_{\overline{p}} = \sum_{\e\in \E(\G)} \Vert [u]_\e \Vert_{L^{p_\e}(0,\ell_{\e})}.$$
 If $p_\e=p$ for all $\e$, we will also write $L^{\overline p}(\G)$ as $L^p(\G)$. 

  We define  $L^{1}(0, T; L^{\overline{p}}(\G))$ as the Banach space of  measurable functions $f : (0,T) \rightarrow  L^{\overline{p}}(\G)$ such that
$$\int_0^T \norm{f(t)}_{L^{\overline{p}}(\G)} \, dt < \infty.$$

Consider
$\widetilde{\mathcal{W}}^{1,\overline{p}}(\G)$ the space of functions in $\G$ such that
$[u]_{\e}\in W^{1,p_\e}(0,\ell_{\e})$
for all $\e\in \E(\G)$, with the Sobolev norm
$$
\norm{u}_{\widetilde{\mathcal{W}}^{1,\overline{p}}(\G)}
:=
\sum_{\e\in \E(\G)} \left(
\norm{[u]_{\e}}_{L^{p_\e}(0,\ell_{\e})}^{p_\e}+
\norm{[u]_{\e}^\prime}_{L^{p_\e}(0,\ell_{\e})}^{p_\e}\right)^{1/p_\e}.
$$
If $u\in\widetilde W^{1,\overline{p}}(\G)$ we denote by $u^\prime$ the function in $L^{\overline{p}}(\G)$ with $[u^\prime]_\e=[u]_\e^\prime$ for all $\e\in\E(\G)$.
The space $\widetilde W^{1,\overline{p}}(\G)$ is a reflexive Banach space when $1<p_\e<+\infty$ for all edge $\e$.

We denote by $C(\V(\G))$ the set of all functions $u$ in~$\G$   such that each $[u]_\e$ is continuous at $[0,\ell_\e]$ and 
$$[u]_{\e_1}(\v)  = [u]_{\e_2}(\v) \quad \hbox{for all} \ \e_1, \e_2 \in \E_\v(\G).$$
We denote this common value at $\v$ as $u(\v)$.

Observe that in the definition of $\widetilde W^{1,\overline{p}}(\G)$ we do not assume the continuity at the vertices.
Set $$\mathcal{W}^{1,\overline{p}}(\G)= \widetilde{\mathcal{W}}^{1,\overline{p}}(\G)\cap C(\V(\G)),$$
equipped with the above Sobolev norm.

 When $p_\e = p$ for all  $\e\in \E(\G)$, we  write $$\widetilde{\mathcal{W}}^{1,\overline{p}}(\G)= \widetilde{W}^{1,p}(\G)$$
and $$\mathcal{W}^{1,\overline{p}}(\G)= W^{1,p}(\G).$$

We will also consider the following space of functions in $\G$:
$$\mathcal{D}(\G):= \{ \varphi   :   [\varphi]_\e \in \mathcal{D}(]0, \ell_\e[) \ \ \forall \, \e \in E(\G) \}.$$

We  define  $ L^{1}(0, T; \mathcal{W}^{1,\overline{p}}(\G))$ as the Banach spaces of  measurable functions $f : (0,T) \rightarrow  \mathcal{W}^{1,\overline{p}}(\G)$ such that
$$\int_0^T \norm{f(t)}_{\widetilde{\mathcal{W}}^{1,\overline{p}}(\G)} \, dt < \infty.$$

 If $I=(a,b)$ is a bounded interval in $\R$, by \cite[Theorem 8.8]{Bre10}, we have that
\begin{equation}\label{RC_001}
\hbox{the injection} \ W^{1,p}(I) \subset C(\overline{I}) \ \hbox{is compact for all} \ 1 < p \leq \infty.
\end{equation}
Consequently we have the following result.

\begin{theorem}\label{compact} 
Let $\G$  be a compact metric graph and $1 < p_\e <  \infty$ for all $\e\in E(\G)$. The injection $ W^{1,\overline{p}}(\G) \subset C(\G)$
is compact.  
\end{theorem}

Moreover, we have the following Poincaré inequality.

\begin{theorem}\label{teopoincare001}
Let $\G$  be a compact metric graph and $1 < p_\e <  \infty$ for all $\e \in E(\G)$. There exists a constant $\lambda_{\G} >0$ such that 
\begin{equation}\label{Poinc0}
\lambda_{\G}
\Vert u \Vert_{L^{\overline{p}}(\G)} \leq \Vert u' \Vert_{L^{\overline{p}}(\G)}   \quad \forall \, u \in W^{1,\overline{p}}(\G):\int_\G
u = 0.
\end{equation}
\end{theorem}

\begin{proof}
We prove the following equivalent formulation:  there exists a constant $C>0$ such that
\begin{equation}\label{Poincpre01}
\Vert u \Vert_{L^{\overline{p}}(\G)} \leq C\left(\Vert u' \Vert_{L^{\overline{p}}(\G)} +\left|\int_\G u\right|\right) 
\quad \forall \, u \in  W^{1,\overline{p}}(\G).
\end{equation}

To prove~\eqref{Poincpre01}  let us argue by contradiction. If it  is not true, then, for any $n\in \mathbb{N}$, there exists $u_n\in W^{1,\overline{p}}(\G)$ such that,
\begin{equation}\label{Poincpre01neg}
\Vert u_n \Vert_{L^{\overline{p}}(\G)} \geq n\left(\Vert (u_n)' \Vert_{L^{\overline{p}}(\G)} +\left|\int_\G u_n\right|\right).
\end{equation}
And we can assume that $\Vert u_n \Vert_{L^{\overline{p}}(\G)}=1$. Hence
$$\Vert (u_n)' \Vert_{L^{\overline{p}}(\G)}\le\frac{1}{n}\ \hbox{ and }\   \left|\int_\G u_n\right|\le\frac{1}{n}.$$ 
Therefore, by Theorem~\ref{compact}, we have that there exists a subsequence, that we denote equal, such that
$$u_n\to u \hbox{ weakly in }  W^{1,\overline{p}}(\G),   \hbox{and in } C(\G).$$
Hence, since  $\Vert (u_n)' \Vert_{L^{\overline{p}}(\G)}\le\frac{1}{n}$, each $[u]_e$ is equal to a constant $C_\e$, but,  since $u $ is continuous,  $C_\e=C$ for all $\e$. Now, from $\displaystyle\left|\int_\G u_n\right|\le\frac{1}{n}$ we get that $C=0$, and this gives a contradiction with  $\Vert u_n \Vert_{L^{\overline{p}}(\G)}=1$.
\end{proof}
   As a particular case of \cite[Lemma 4.2]{AndIgMazTo06}  we have the following result.  

 \begin{lemma}\label{Poinc} Let $\gamma$ be a maximal monotone graph in $\R^2$ such that ${\rm Ran}(\gamma) = \R$. Let $u_n \in W^{1,p}(0,\ell)$, $1 < p < \infty$ and $v_n \in \gamma(u_n)$. Suppose that  there exists $M>0$ such that
 $$\int_0^{\ell} \vert v_n \vert dx \leq M \quad \forall \, n \in \N.$$
 Then, there exists $C >0$ such that
 $$\Vert u_n \Vert_{L^p(0,\ell)} \leq C( \Vert u^{\prime}_n )\Vert_{L^p(0,\ell)} +1)  \quad \forall \, n \in \N.$$
\end{lemma}

\subsection{Nonlinear semigroups}
\label{sec:semigroups}
Let $X$ be  a Banach space with norm $\norm{\cdot}$.
An operator $A$ on $X$ is a possibly nonlinear
and multi-valued mapping $A : X\to 2^{X}$. It is standard to identify
an operator $A$ on $X$ with its graph in $X \times X$:
\begin{displaymath}
  A=\Big\{(u,v)\in X\times X: v\in Au\Big\}.
\end{displaymath} 
The set
$D(A):=\{u\in X\,\vert\,Au\neq \emptyset\}$ is called the
 domain  of $A$, and
$\textrm{Ran}(A):=\bigcup_{u\in D(A)}Au$ the  range 
of $A$.

\begin{definition}\label{def:quasi-accretive}\rm
  An operator $A$ on $X$ is called
 accretive  on $X$
  if for every $(u,v)$,
  $(\hat{u},\hat{v})\in A$ and every $\lambda\ge 0$,
  \begin{displaymath}
    \norm{u-\hat{u}}\le \norm{u-\hat{u}+\lambda (v-\hat{v})}.
  \end{displaymath}
   An accretive operator $A$ is called {\it $m$-accretive}  if for some (equivalently, for all) $\lambda>0$, the following \emph{range condition} holds true:
  \begin{equation}
    \label{eq:29}
    \hbox{Ran}(I+\lambda A)=X.
  \end{equation}
\end{definition}

If $A$ is  $m$-accretive operator on a Banach space $X$, then, by the
 Nonlinear Semigroup Theory (see, e.g.,~\cite[Theorem~6.5]{Benilanbook},
or~\cite[Corollary~4.2]{Bar10}), the first-order Cauchy problem
\begin{equation}
  \label{eq:30}
  \begin{cases}
    \displaystyle\frac{d u}{dt}+A(u(t))\ni f(t) &\text{on $(0,T)$,}\\[8pt]
     \displaystyle u(0)=u_{0},
  \end{cases}
\end{equation}
is well-posed for every $u_{0}\in \overline{D(A)}$  and
$f\in L^{1}(0,T;X)$  in the following \emph{mild sense}.

\begin{definition}\label{def:mild}\rm 
  For given $u_{0}\in \overline{D(A)}$ and
  $f\in L^{1}(0,T;X)$, a function $u\in C([0,T];X)$ is called a
 mild solution of Cauchy problem~\eqref{eq:30}  on $[0,T]$ if $u(0)=u_{0}$
and, for every
$\epsilon>0$, there is a partition
$0=t_{0}<t_{1}<\cdots < t_{N}=T$, and $f_1, \ldots , f_N \in X$, and a  step
  function
  \begin{displaymath}
    u_{\epsilon,N}(t)=u_{0}\,\1_{\{t=0\}}(t)+\sum_{i=1}^{N}u_{i}\,\1_{(t_{i-1},t_{i}]}(t)
    \quad\text{defined for  $t\in [0,T]$,}
  \end{displaymath}
 such that
  \begin{align*}
   &t_{i}-t_{i-1}<\epsilon\quad\text{ for all $i=1,\dots,N$,}\\[8pt]
    &\sum_{i=1}^{N}\int_{t_{i-1}}^{t_{i}}\norm{f(t)-f_{i}}\,dt<\epsilon,\\[8pt]
   & \frac{u_{i}-u_{i-1}}{t_{i}-t_{i-1}}+A u_{i}\ni
                   f_{i}
                   \quad\text{ for all $i=1,\dots,N$,}\\
  and \hspace{1cm}& \hfill \\
       & \sup_{t\in [0,T]}\norm{u(t)-u_{\epsilon,N}(t)}<\epsilon.
  \end{align*}
\end{definition}

\begin{theorem}[ Crandall-Liggett Theorem]\label{EUm-accretive} Let $A$ be an operator in $X$, $f \in L^1(0, T; X)$ and $x_0
\in \overline{D(A)}$. If $A$ is m-accretive, then the problem \eqref{eq:30}
has a unique mild solution $u$ on  $[0, T]$.
\end{theorem}

  Ph. Bénilan in \cite{B} (see also \cite{Benilanbook}) introduced the following concept of solution for Problem~\eqref{eq:30}.
\begin{definition}\label{intSol}\rm  A function $u \in C(0, T; X)$ is an {\it integral solution} of problem \eqref{eq:30} if $u(0) = u_0$ and, for all $(x,y) \in A$, the following inequality holds in $\mathcal{D}^{\prime}(]0,T[)$:
\begin{equation}\label{E1IntSol}\frac{d}{dt} \Vert u(t) - x \Vert \leq [u(t) - x, f(t) - y],
\end{equation}
  where  the bracket $[.,.]$ is defined as
$$[x,y]:= \lim_{\lambda \to 0} \frac{\Vert x + \lambda y \Vert - \Vert x \Vert}{\lambda}.$$
\end{definition}
Moreover  the following result was proved. 
 \begin{theorem}[\cite{B}, \cite{Benilanbook}]\label{mildInt} If $A$ is accretive, then
 $$u \ \hbox{is a mild solution of} \ \eqref{eq:30} \iff u \ \hbox{is an integral solution of} \ \eqref{eq:30}.$$
\end{theorem}

\begin{definition}{\rm Let $X$ be a Banach lattice and $S : D(S) \subset
X \rightarrow X$. We say that $S$ is a  
T-contraction  if
$$\Vert (Su - Sv)^+ \Vert \leq \Vert (u - v)^+ \Vert \ \ \
\hbox{for} \ u, v \in D(S).$$

Let $A$ be an operator in a Banach lattice $X$. We say that $A$ is  
T-accretive   if
$$\Vert (u - \hat u)^+ \Vert \leq \Vert (u - \hat u + \lambda (v - \hat v))^+
\Vert \ \ \ \hbox{for} \ (u, v), (\hat u, \hat v) \in A \ \
\hbox{and} \ \lambda > 0.$$ }
\end{definition}

It is clear that a T-contraction is order-preserving; also if
$A$ is a T-accretive operator then its resolvents $(I + \lambda
A)^{-1}$ are single-valued  and T-contractions.

 The mild solutions of the abstract
 Cauchy problems associated with T-accretive operators satisfy
 a comparison and contraction principle.

 \begin{theorem}\label{ContPrinTT} Let $A$
 be a $m$-T-accretive operator in a Banach lattice $X$. If $f, \hat f \in L^1(0, T; X)$, and $u$, $\hat u$
  are mild solutions of $u' + A u \ni f$ and $\hat{u}' + A \hat{u} \ni
  \hat{f}$ on $[0, T]$, then, for $0\leq s \leq t \leq T$,
  $$\Vert (u(t) - \hat{u}(t))^+ \Vert \leq \Vert (u(s) - \hat{u}(s))^+
  \Vert + \int_s^t \norm{ (f(\tau) - \hat{f}(\tau))^+} d \tau.$$
\end{theorem}

 Following~\cite{BenCra91}, consider the following relation for $u,v \in L^1(\G)$: 
$$u\ll v   \hbox{ if }   \int_\G (u - k)^+ \leq \int_\G (v - k)^+   \hbox{ and }    \int_\G (u +k)^- \leq \int_\G (v + k)^-    \hbox{ for any } \ k >0.$$
From~\cite{BenCra91}  we have  the following result .
\begin{proposition}\label{inq1BC}  The following properties hold true:
\item[\ \ (1)] $\displaystyle \int_\G v\chi_{\{u<-h\}}dx\le 0\le  \int_\G v\chi_{\{u>h\}}dx$ for all $h>0$ iff 
$u\ll u+\lambda v$ for all $\lambda>0$. 
\item[\ \ (2)] If  $u \ll v$  then $\Vert u \Vert_p \leq \Vert v \Vert_p $ for any $p \in [1,+\infty]$.
\item[\ \ (3)] If $v \in L^1(\G)$, then $\{ u \in L^1(\G)  :  u \ll v \}$ is a weakly compact subset of $L^1(\G)$.
\item[\ \ (4)] If $\{u_n\}_n$ is a sequence in $L^1(\G)$ with $u_n\ll u$ for all $n$, and $u_n\to u$ weakly in $L^1(\G)$, then it converges strongly.
\end{proposition}

\section{Doubly Nonlinear  Equations  on Metric Graphs. }\label{sec23}

\subsection{Results on a bounded interval}
Let $1 < p < \infty$ be. Let us begin with some known facts for the $p$-Laplacian problem on a bounded interval,  which can be seen as the trivial metric graph of one edge.  

Given an interval $[0, \ell]$,  $g\in L^1(0,\ell)$  and $a,b\in \mathbb{R}$, we say that $u \in W^{1,p}(0,\ell)$ is a {\it weak solution} of the Neumann problem 
 \begin{equation}\label{NeumannI}
 \left\{ \begin{array}{ll}  -\Delta_pu = g \quad \hbox{in} \ (0,\ell), 
 \\[10pt]  -(\vert u' \vert^{p-2} u')(0) = a, 
  \\[10pt]  (\vert u' \vert^{p-2} u')(\ell) =b,
 \end{array} \right.
 \end{equation}
 if
 \begin{equation}\label{weakform}
\int_0^\ell |u'|^{p-2}u' \varphi'\,dx
= \int_0^\ell g\,\varphi\,dx
+ a\,\varphi(0)+b\,\varphi(\ell)
\quad
\forall \varphi\in W^{1,p}(0,\ell).
\end{equation}
Note that, since $W^{1,p}(0, \ell) \subset C([0, \ell])$, we have that the last integral is finite. And by taking $\varphi=1$ we need the compatibility condition
 $$\int_0^\ell g(x)dx+a+b=0.$$

 \begin{remark}\label{regularity}\rm  1. For $g\in L^1(0,\ell)$,  let $u\in W^{1,p}(0,\ell)$ be a  weak solution
 of problem~\eqref{NeumannI}. Then, $$\left(|u'|^{p-2}u'\right)'=g\quad\hbox{in } D'(]0,\ell[).$$ This implies that 
$
|u'|^{p-2}u' \in W^{1,1}(0,\ell)$, and hence, by Sobolev embedding we can assume that $|u'|^{p-2}u'\in C([0,\ell])$. 
Therefore, $u$ has a representative in $C^1([0,1])$.  
Moreover, we have that
$$\int_y^zg(x)dx=(|u'|^{p-2}u')(y)-(|u'|^{p-2}u')(z)\quad\forall y,z\in[0,\ell].$$

2. For $z,w \in W^{1,1}(0,\ell)$, we   have the following integration by parts formula:
\begin{equation}\label{IPR}
\int_0^\ell z'(x)  w(x) \, dx = - \int_0^\ell z(x)  w'(x) \, dx + z(\ell)w(\ell) - z(0) w(0).
\end{equation}
\hfill$\blacksquare$
\end{remark}

 Given   $\gamma:\R \to \R$ continuous and increasing with $\gamma(\mathbb{R})=\mathbb{R}$ and $\gamma(0)=0$,  $1 < p < \infty$, $f \in L^1(0,\ell)$ and $a,b \in \R$,   consider the problem
$$(P_{p,a,b}^{\gamma,f})\quad\left\{\begin{array}{lll}
v -\Delta_{p} u =f \quad\hbox{in } (0,\ell),\\
\\
v =\gamma(u)\quad\hbox{a.e. in } (0,\ell),
\\ \\
-\left(\vert u'\vert^{p-2} u' \right)(0)=a,\\ \\ \left(\vert u'\vert^{p-2} u' \right)(\ell)=b.
\end{array}\right.$$
We give the following concept of solution for this problem.
\begin{definition} \rm We say that   $ v\in L^1(0, \ell)$  is a {\it weak solution} of problem $(P_{p,a,b}^{\gamma,f})$ if   there exists $u\in W^{1,p}(0,\ell)$, with  $v=\gamma(u)$ a.e. in $(0, \ell)$, such that
\begin{equation}\label{Weak1NN}\int_0^{\ell} v \phi   + \int_{0}^{\ell}|u'|^{p-2}u' \phi^{\prime}    = \int_0^{\ell} f \phi  +     b\phi(\ell) + a\phi(0) \quad \forall \, \phi \in W^{1,p}(0, \ell). \end{equation}
\end{definition}

  If $v$ is a weak solution to the problem $(P_{p,a,b}^{\gamma,f})$, we have the following  mass balance property,
\begin{equation}\label{test1}
\int_0^{\ell} v =\int_0^{\ell} f  +  b + a,
\end{equation}
  and we can assume that $u$ is $C^1([0,\ell])$ (see~Remark~\ref{regularity}).

From~\cite[Theorem 3.6]{AndIgMazTo07} we have the following result.

\begin{theorem}\label{lemma1}  Let $\gamma:\R \to \R$ be continuous and increasing with $\gamma(\mathbb{R})=\mathbb{R}$ and $\gamma(0)=0$, and   $1 < p < \infty$. For $f \in L^1(0,\ell)$ and $a,b \in \R$, there exists a unique weak solution to problem $(P_{p,a,b}^{\gamma,f})$.
\end{theorem}

 \subsection{The elliptic problem on a metric graph.} Let  $\G$ be a   connected and compact metric graph  (without loops and without multiple edges).
 For each   $\e\in \E(\G)$, let $\gamma_\e:\R \to \R$ be continuous and increasing with $\gamma_\e(\mathbb{R})=\mathbb{R}$ and $\gamma_\e(0)=0$. 
   We write $\overline{\gamma}$ to represent  the collection of $(\gamma_\e)_{\e \in \E(\G)}$ and 
we  write $[v]_\e=\gamma_\e([u]_\e)$  for all $\e\in  \E(\G)$ as
 $$v= \overline{\gamma}(u).$$
Take also $\overline{p}$ as in the previous Subsection~\ref{sec1747} with $1<p_\e<+\infty$ for all $\e\in \E(\G)$.

 We   define the upwind values of a general function $z:\G\to\R$, interpreted as fluxes on $\G$ at vertices $\v\in\V(\G)$, as follows. Let   $\v\in\V(\G)$ and $\e\in\E_\v(\G)$ be, we define
\begin{align}\label{eq:upwind_values}
\{z\}_\e(\v) :=
\begin{cases}
+[z]_\e(\ell_\e),\quad&\text{if }\v=\vf_\e,\\
-[z]_\e(0),\quad&\text{if }\v=\vi_\e,
\end{cases}
\end{align}
whenever these values exist in the sense of traces, e.g.,  if $z\in \widetilde W^{1,1}(\G)$. 

   From \eqref{IPR}, we have the the following {\it Green's formula} on a metric graph: 
\begin{equation}\label{GREENF0}
  \int_{\G}z' w\, dx  = -\int_{\G}z w'\, dx+ \sum_{\v\in \V(\G)} \left(\sum_{\e \in\E_{\v}(\G)} 
  \{z\}_\e(\v)\right)w(\v)\quad\hbox{for }  z,w\in W^{1,1}(\G).
\end{equation}

  Let  $  f\in L^{1}(\G)$ be,  and $ \omega=\{ \omega_\v: \v \in \V(\G) \} \subset \R$ (we will write in this simple way the set of fluxes $\omega_\v$ assigned to each vertex $\v$). Our   aim in this section is to solve the elliptic problem
\begin{equation}\label{24111058} 
 (P_{\omega}^f)\quad\left\{\begin{array}{l}
[v]_\e-\Delta_{p_\e} [u]_\e=[f]_\e\quad\hbox{in } (0,\ell_\e),\ \forall\, \e\in \E(\G),\\
\\{}
[v]_\e=\gamma_\e([u]_\e)\quad\hbox{a.e.},\ \forall\, \e\in \E(\G),
\\ \\ 
\hbox{$u$ continuous,}
\\ \\
\displaystyle \partial^{\overline{p}}_\nu u(\v)= {\omega_\v}   \quad\hbox{for all } \v\in \V(\G).
\end{array}\right.
\end{equation}
Here,  
 \begin{equation}\label{GREENF0op} \partial^{\overline{p}}_\nu u(\v):=\displaystyle\sum_{\e \in\E_{\v}(\G)} 
  \{z_{\overline{p}} \}_\e(\v), 
  \end{equation}
 being $z_{\overline{p}}$ defined as  
 $$[z_{\overline{p}}]_\e=  |[u]_\e'|^{p_\e-2}[u]_\e',$$
so that,   
$$\partial^{\overline{p}}_\nu u(\v)= {\omega_\v}$$  is a {\it generalized Neumann-Kirchhoff flux condition},
which of course includes the {\it standard
  Neumann-Kirchhoff condition}
 $$\partial^{\overline{p}}_\nu u(\v)
 =0 \quad\hbox{for all } \v\in \V(\G),$$
 by taking  $\omega_\v=0$ for all $\v$.

{\it A priori}, \eqref{GREENF0op}~should be read as a formal expression that has sense joint with Definition~\ref{ladeg1834dom}. But, in fact, it is not a formal expression because of regularity (see Remark~\ref{KirchTrue}).

  When the index of a vertex $\v$ is $2$, and the edges to which it belongs have the same diffusion rule, it is evident that if the flux $\omega_\v=0$ we can join these edges in an unique one, and forget this vertex. This is no the case if the diffusion of each edge is different, or if, even being the same, $\omega_\v\neq0$.

 As mentioned in the introduction, the study of this  problem  will allow us to solve, using the Crandall-Liggett Theorem, the evolution problem~\eqref{evpbintro1711}, but this elliptic problem, and the way we prove existence, is interesting by itself.

 Green's formula~\eqref{GREENF0} induces the following definition.

\begin{definition}\label{ladeg1834dom}\rm We say that $v\in L^{1}(\G)$ is a {\it weak solution} of problem $(P_{\omega}^f)$ if  there exists $u\in \mathcal{W}^{1,\overline{p}}(\G)$, with $v= \overline{\gamma}(u)$ a.e in $\G$, such that
\begin{equation}\label{Weak1}  \int_\G v\varphi\, dx+ \sum_{\e \in\E(\G)}\int_{0}^{\ell_\e}|[u]_\e'|^{p_\e-2}[u]_\e' [\varphi]_\e'\,   dx= \int_\G f \varphi \,dx + \sum_{\v\in \V(\G)}\omega_\v\varphi(\v) \quad \forall \, \varphi \in \mathcal{W}^{1,\overline{p}}(\G). \end{equation}
\end{definition}

 \begin{remark}\label{KirchTrue} \rm   Note that given $\e \in E(\G)$, if we take in \eqref{Weak1} as test function $\varphi$ such that $[\varphi]_\e = \phi$, with $\phi \in \mathcal{D}(]0, \ell_\e[)$ and $[\varphi]_{\e'} = 0$ for all $\e' \not= \e$, then we obtain that
\begin{equation}\label{distribu} 
[v]_\e  = (\vert [u]_\e' \vert^{p_\e -2} [u]_\e')' + [f]_\e \quad \hbox{in} \quad \mathcal{D}^{\prime}(]0, \ell_\e[). 
\end{equation}  As a consequence   we can consider  that $[u]_\e\in C^1([0,\ell_\e])$ (Remark~\ref{regularity}), and we do it along the paper.
Moreover, by Green's formula, the  Kirchhoff condition
 $$  \partial^{\overline{p}}_\nu u (\v)=   \omega_\v \quad \forall \v \in V(\G)$$
is   satisfied point--wise, and we will have the next Proposition~\ref{lemma2}.

 Observe also that, taking   as test function $\varphi =1$ in \eqref{Weak1}, we get the following {\it mass balance property}
\begin{equation}\label{MBC1}
\int_\G v = \int_\G f   + \sum_{\v\in \V(\G)}\omega_\v.
\end{equation}
\hfill$\blacksquare$
 \end{remark}

  Let us point out that that given a problem  $(P_{\omega}^f)$ on the metric graph $(\G = (\V(G), \E(\G))$, we can change the parametrization of the graph changing the function $f$. For instance, if in an edge $\e$ we have  $\v_1=\vi_\e$, $\v_2=\vf_\e$, and we change the parameterization to  $\v_2=\vi_\e$ and $\v_1=\vf_\e$, then we get the same problem by changing $f$ to $\tilde{f}$  where $[\tilde{f}]_\e(x):= [f]_\e(x - \ell_\e)$.  Observe that if $u$ is the solution for the original parametrization and $\tilde u$ is the solution for the new parametrization, then $[\tilde u]_\e(x)=[u]_\e(\ell_\e-x)$ and $[\tilde v]_\e(x)=[v]_\e(\ell_\e-x)$.

 \begin{proposition}\label{lemma2} If $v$ is a weak solution  to problem $(P_{\omega}^f)$, then, for all $\e \in E(\G)$, there exist $a_{\e},b_{\e} \in \R$  such that $[v_\e]$  is weak solution to problem ($P_{p_\e,a_\e,b_\e}^{\gamma_\e,[f]_\e}$).   Moreover, for any $\v \in V(\G)$, we have
\begin{equation}\label{wwg1333}\sum_{\e:\vi_\e=\v}a_\e+\sum_{\e:\vf_\e=\v}b_\e=\omega_\v.
\end{equation}
\end{proposition}

\begin{proof}  
 By the previous remark we have that $[v]_\e$ is a weak solution to ($P_{p_\e,a_\e,b_\e}^{\gamma_\e,[f]_\e}$) where, for $[u]_\e=\gamma_\e{}^{-1}([v]_\e)$,
$$a_\e=-\left(\vert [u]_\e'\vert^{p-2} [u]_\e' \right)(0) \ \hbox {and } \ b_\e=\left(\vert [u]_\e'\vert^{p-2} [u]_\e' \right)(\ell_\e).$$ Now, for $\varphi\in \mathcal{W}^{1,\overline{p}}(\G)$,  consider the corresponding equations~\eqref{Weak1NN} for each $\e\in E(\G)$ with $\phi=[\varphi]_\e$. Then, adding such equations and comparing the result with~\eqref{Weak1}, we get
$$\sum_{\v\in \V(\G)}\omega_\v\varphi(\v)=\sum_{\v\in \V(\G)}\left(\sum_{\e:\vi_\e=\v}a_\e+\sum_{\e:\vf_\e=\v}b_\e\right)\varphi(\v).
$$
Since this is true for any $\varphi\in \mathcal{W}^{1,\overline{p}}(\G)$,~\eqref{wwg1333} follows.
\end{proof}

\begin{remark}\label{defdelplaplacianoenG}
\rm By defining the following $\overline{p}$-{\it Laplacian operator} $\Delta_{\overline{p}}^{\G,\omega}$ on functions $u\in \mathcal{W}^{1,\overline{p}}(\G)$ as:
$$\int_\G-\Delta_{\overline{p}}^{\G,\omega}u (x)\varphi(x)=\sum_{\e \in\E(\G)}\int_{0}^{\ell_\e}|[u]_\e'|^{p_\e-2}[u]_\e' [\varphi]_\e'\,   dx-\sum_{\v\in \V(\G)}\omega_\v\varphi(\v) \quad \forall \, \varphi \in \mathcal{W}^{1,\overline{p}}(\G),$$
Problem~\eqref{24111058} can be written as
$$  (P_{\omega}^f)\quad\left\{\begin{array}{l}
v-\Delta^{\G,\omega}_{\overline{p}}u=f\quad\hbox{in } \G,\\
\\{}
v=\overline{\gamma}(u)\quad\hbox{in } \G,
\\ \\
\hbox{$u$ continuous,}
\end{array}\right.$$
since  the definition of $\Delta^{\G,\omega}_{\overline{p}}$  includes the generalized Neumann-Kirchhoff flux condition.
\hfill$\blacksquare$
 \end{remark}

 We have uniqueness of weak solutions. 

 \begin{theorem}\label{teocontprin001} Let  $\G$ be a  connected and compact metric graph.
  Let $1 < p_\e < \infty$ be and   $\gamma_\e:\R\to \R$ be a continuous and increasing function with $\gamma_\e(\mathbb{R})=\mathbb{R}$ and $\gamma_\e(0)=0$ for all~$\e \in E(\G)$.
   For $v_i$ weak solutions  to $(P_{\omega^i}^{f_i})$, $i=1,2$, we have  
$$\int_{\G}(v_1-v_2)^+\le \int_{\G}(f_1-f_2)^+ + \sum_{\v\in\V(\G)}(\omega^1_\v-\omega^2_\v)^+\,.$$
\end{theorem}

\begin{proof} Let $v_i$  be weak solutions to $(P_{\omega}^f)$, $i=1,2$, and take $ u_i$ with $v_i= \overline{\gamma}(u_i)$.
    Take $\varphi=\frac{1}{k}T_k(u_1-u_2)^+$ as test function in the definition of solution for $i=1$ and $i=2$, and subtract the corresponding equations to get, by misleading the non-negative terms involved with $|[u_i]_\e'|^{p_\e-2}[u_i]_\e'$, 
    $$\int_\G(v_1-v_2)^+\frac{1}{k}T_k(u_1-u_2)^+\le \int_\G(f_1-f_2)^++\sum_{\v\in \V(\G)}(\omega^1_\v-\omega^2_\v)^+.$$
    Hence, taking limits as $k\to +\infty$, we get the result.
\end{proof}

 From now on, and for the sake of simplicity, we will say \lq\lq let $v=\overline{\gamma}(u)$ be a weak solution of $(P_{\omega}^f)$\rq\rq\ to give also implicitly the function $u$.

 In order to proof existence of weak solutions, the main problem is to the get the continuity at vertices.   To do that we will use the following  lemmas. In them we compare a solution to the problem  $(P_{\omega}^f)$  with the solution of a perturbed problem $(P_{\omega^\epsilon}^f)$ where the   fluxes on vertices are modified. Moreover, we establish continuity properties of these perturbed problems with respect to the perturbation.
We start by considering the simple case of a single edge, which, along with its proof and Proposition~\ref{lemma2}, is essential for the remainder of the arguments.

  \begin{lemma}\label{lemma3} Assume $\gamma:\mathbb{R}\to \mathbb{R}$ is a continuous and increasing function with $\gamma(\mathbb{R})=\mathbb{R}$ and $\gamma(0)=0$. We have the  statements:
 \begin{itemize}
 \item[1.]  Let $v_\epsilon=\gamma(u_\epsilon)$ be weak solution to problem ($P_{p,a,b+\epsilon}^{\gamma,f}$),  $\epsilon \in \mathbb{R}$. Then:
\\ (i)  For $\epsilon>0$, we have have that
 $$\begin{array}{l}\displaystyle 
 u_0 \leq u_\epsilon,
 \\ \\
 \displaystyle u_0(\ell) <u_\epsilon(\ell);
\end{array} $$ 
  for $\epsilon<0$, we have that 
  $$\begin{array}{l}\displaystyle u_\epsilon\le u_0,
   \\ \\
 \displaystyle 
u_\epsilon(\ell)<u_0(\ell).
\end{array} $$ 
 \\ (ii) Moreover, $\epsilon\mapsto u_\epsilon(x)$ is continuous in $\mathbb{R}$ for all fixed $x\in[0,\ell]$, and 
\begin{equation}\label{ran001}
\begin{array}{l}
\displaystyle \lim_{\epsilon \to  +\infty} u_\epsilon(\ell) =  +\infty,
 \\ \\
\displaystyle \lim_{\epsilon \to  -\infty} u_\epsilon(\ell) =  -\infty.
\end{array} \end{equation}

  \item[2.] Let $v_\epsilon=\gamma(u_\epsilon)$ be weak solution to problem ($P_{p,a+\epsilon,b}^{\gamma,f}$), with $\epsilon \in \mathbb{R}$. Then:
  \\ (i)  For $\epsilon>0$, we have have that
 $$\begin{array}{l}\displaystyle 
 u_0 \leq u_\epsilon,
 \\ \\
 \displaystyle u_0(0) <u_\epsilon(0);
\end{array} $$ 
  for $\epsilon<0$, we have that 
  $$\begin{array}{l}\displaystyle u_\epsilon\le u_0,
   \\ \\
\displaystyle 
u_\epsilon(0)<u_0(0).
\end{array} $$ 
 \\ (ii) Moreover, $\epsilon\mapsto u_\epsilon(x)$ is continuous in $\mathbb{R}$ for all fixed $x\in[0,\ell]$, and 
\begin{equation}\label{ran001dos001}
\begin{array}{l}
\displaystyle \lim_{\epsilon \to  +\infty} u_\epsilon(0) =  +\infty,
 \\
\\
\displaystyle \lim_{\epsilon \to  -\infty} u_\epsilon(0) =  -\infty.
\end{array} \end{equation}

\item[3.] Let   $v_\epsilon=\gamma(u_\epsilon)$ be weak solution of $(P_{p,a-\epsilon,b+\epsilon}^{\gamma,f})$ for $\epsilon\in \mathbb{R}$.
\\
(i) For $\epsilon>0$ we have
$$u_\epsilon(\ell)> u_0(\ell)$$ and
$$u_\epsilon(0)< u_0(0).$$
 For $\epsilon<0$ we have
$$u_\epsilon(\ell)< u_0(\ell),$$ and
$$u_\epsilon(0)> u_0(0).$$
(ii) Moreover, $\epsilon\mapsto u_\epsilon(\ell)$ and $\epsilon\mapsto u_\epsilon(0)$ are continuous in $\mathbb{R}$. 
 \end{itemize}
 \end{lemma}
 
 \begin{proof}  {\bf Proof of Point 1}.    Take  $\epsilon>0$.  The monotonicity  $ u_0 \leq u_\epsilon$ is consequence of Theorem~\ref{teocontprin001} and the monotonicity of $\gamma$. 
 Let us now proof that $u_0(\ell) <u_\epsilon(\ell)$. 
By the mass balance property~\eqref{test1} we have
  \begin{equation}\label{epsilon1}
\int_0^{\ell} (v_\epsilon - v_0) = \epsilon.
  \end{equation}
We also have 
$$v_0 = \left(|u_0'|^{p-2}u_0' \right)' + f,$$
and
  $$v_\epsilon = \left(|u_\epsilon'|^{p-2}u_\epsilon' \right)' + f. $$
Hence, for $0 \leq x \leq \ell$, having in mind the monotonicity and \eqref{epsilon1}, we have
$$0\le \int_0^x \left(  |u_\epsilon'|^{p-2}u_\epsilon'  -  |u_0'|^{p-2}u_0'   \right)'
= \int_0^x (v_\epsilon - v_0) \nearrow \epsilon \quad \hbox{as } x\nearrow \ell.$$
Hence,
$$0\le  (|u_\epsilon'|^{p-2}u_\epsilon') (x) -  (|u_0'|^{p-2}u_0') (x) - \left((|u_\epsilon'|^{p-2}u_\epsilon') (0) -  (|u_0'|^{p-2}u_0') (0) \right) \nearrow \epsilon \quad \hbox{as } x\nearrow \ell.$$
Now, 
$$(|u_\epsilon'|^{p-2}u_\epsilon') (0) = a \quad \hbox{and} \quad (|u_0'|^{p-2}u_0') (0) = a,$$
therefore, we have
\begin{equation}\label{17531510}
0 \leq (|u_\epsilon'|^{p-2}u_\epsilon') (x) -  (|u_0'|^{p-2}u_0') (x) \nearrow \epsilon \quad \hbox{as } x\nearrow \ell.
\end{equation}
Then, since the function $\varrho(r):= \vert r \vert^{p-2}r$ is   increasing for $p >1$, we obtain, from the first inequality in~\eqref{17531510},
$$u_0'(x) \leq u_\epsilon'(x),$$
 and integrating here, for $0\le x\le y\le \ell,$ 
$$u_0(y) - u_0(x)\leq u_\epsilon(y) -  u_\epsilon(x),$$
so,    
\begin{equation}\label{cry001}
0 \leq  u_\epsilon(x)  - u_0(x)\leq u_\epsilon(y) - u_0(y)\le u_\epsilon(\ell) - u_0(\ell).
\end{equation}
 Hence, if $u_\epsilon(\ell) - u_0(\ell)=0$, then $u_\epsilon=u_0$, which is false. Then we have $u_0(\ell) <u_\epsilon(\ell)$.

 The proof for $\epsilon<0$ is similar, and we have proved (i). 
 
 Let us now prove  (ii). By~\eqref{epsilon1}, there exists a  sequence $\epsilon_n\to 0$ such that $v_{\epsilon_n}\to v_0$ almost everywhere. Then, by~\eqref{cry001}, 
\begin{equation}\label{3.12.2}u_\epsilon(x)\to u_0(x)\quad\hbox{as }\epsilon\to 0^+,
\end{equation}
for any $x\in[0,\ell)$. 
Let us see that $$u_\epsilon(\ell)\to u_0(\ell)\quad\hbox{as }\epsilon\to 0.$$ 
From~\eqref{17531510} we also have: 
\begin{equation}\label{conv0003}
u_\epsilon'(x)\le \varrho^{-1}(\varrho(u_0'(x))+\epsilon).
\end{equation}
 Hence, integrating between $0$ and $\ell$,
 $$ u_\epsilon(\ell) - u_\epsilon(0)\le \int_0^\ell \varrho^{-1}(\varrho(u_0'(x))+\epsilon)dx,$$
 and therefore,
 \begin{equation}\label{12061610}
0\le u_\epsilon(\ell) - u_0(\ell)\le \int_0^\ell \varrho^{-1}(\varrho(u_0'(x))+\epsilon)dx+u_\epsilon(0)-u_0(\ell).
\end{equation}
By the  Monotone Convergence Theorem, 
 $$\lim_{\epsilon\to 0}\int_0^\ell \varrho^{-1}(\varrho(u_0'(x))+\epsilon)dx=\int_0^\ell u_0'(x)dx=u_0(\ell)-u_0(0),$$
we moreover have shown in~\eqref{3.12.2} that  $$u_\epsilon(0)\to u_0(0),$$   hence
 $$ 
 \displaystyle\lim_{\epsilon\to 0}\int_0^\ell \varrho^{-1}(\varrho(u_0'(x))+\epsilon)dx+u_\epsilon(0)-u_0(\ell)=0,
 $$
 then, from~\eqref{12061610},
 $$u_\epsilon(\ell)\to u_0(\ell)\quad\hbox{as }\epsilon\to 0^+.$$
 
The limit when  $\epsilon\to 0^-$ is similar.   Hence, we have  the continuity of $\epsilon\mapsto u_\epsilon(\ell)$ at $0$. Therefore,  fixed $x\in[0,\ell]$, having in mind~\eqref{cry001},   we have  the continuity of $\epsilon\mapsto u_\epsilon(x)$ at $0$. Now, it is easy to see that we have the continuity of $\epsilon\mapsto u_\epsilon(x)$ at any point~$\epsilon$.

Let us finally prove the  statements in~\eqref{ran001}.   Take $\epsilon>0$.      By \eqref{epsilon1},
$$\int_0^\ell v_\epsilon(x)dx=\int_0^\ell v_0(x)dx+\epsilon,$$ and from here, there exists $0\le x_\epsilon\le\ell$ such that 
$$\lim_{\epsilon\to+\infty}v_\epsilon(x_\epsilon)=+\infty,$$
hence, 
$$\lim_{\epsilon\to+\infty}u_\epsilon(x_\epsilon)=+\infty.$$
Consequently,   since  by~\eqref{cry001},
$$  u_\epsilon(x_\epsilon) - u_0(x_\epsilon)\le u_\epsilon(\ell) - u_0(\ell),$$
and $u_0(x_\epsilon)$ is bounded, we get 
$$\lim_{\epsilon\to+\infty}u_\epsilon(\ell)=+\infty.$$
The other limit  follows in the same way.

The proof of {\bf Point 2.} is similar.

 Let us now prove {\bf Point 3.}  Take $\epsilon>0$. 
In this case   \begin{equation}\label{1746411new0}
\int_0^\ell (v_\epsilon-v_0)=0.
\end{equation}

 Since
$$(|u_\epsilon'|^{p-2}u_\epsilon')(\ell)-(|u_0'|^{p-2}u_0')(\ell)=\epsilon,$$
 and $x\mapsto (|u_\epsilon'|^{p-2}u_\epsilon')(x)-(|u_0'|^{p-2}u_0')(x)$ is continuous, we have that
$$(|u_\epsilon'|^{p-2}u_\epsilon')(x)-(|u_0'|^{p-2}u_0')(x)>0$$
on an interval $(x_\epsilon,\ell)$. Let us consider that this interval is the  the largest one (with $\ell$ fixed) for such property.  Let us now  consider the following two cases, (a) and (b).

\noindent (a) If $x_\epsilon=0$ then, since
  $$(|u_\epsilon'|^{p-2}u_\epsilon')(0)-(|u_0'|^{p-2}u_0')(0)=\epsilon,$$  
  we have that
$$(|u_\epsilon'|^{p-2}u_\epsilon')(x)-(|u_0'|^{p-2}u_0')(x)>0\quad\forall x\in[0,\ell].$$
Hence 
\begin{equation}\label{arre002new} u_0'(x) <u_\epsilon'(x)\quad\forall x\in[0,\ell],
\end{equation}
and integrating  
\begin{equation}\label{arre001new}   u_\epsilon(x)-u_0(x)< u_\epsilon(y)-u_0(y)\quad\forall 0\le x<y\le\ell.
\end{equation}
Now, if  $u_\epsilon(\ell)\le u_0(\ell)$,  by~\eqref{arre001new},
$$u_\epsilon(x)-u_0(x)\le 0\quad\forall 0\le x \le\ell,$$
and also, since $\gamma$ is increasing, 
$$v_\epsilon(x)-v_0(x)\le 0 \quad\forall 0\le x \le\ell.$$
Then, by~\eqref{1746411new0},
we have that $v_\epsilon=v_0$, which is impossible. Therefore 
$$u_\epsilon(\ell) >u_0(\ell).$$
And similarly we have that 
$$u_\epsilon(0) <u_0(0).$$
Let $\widehat x_\epsilon>0$  be such that $u_\epsilon(x)-u_0(x)>0$ for  $\widehat x_\epsilon<x\le \ell$
and  
\begin{equation}\label{mue02001}u_\epsilon(\widehat x_\epsilon)=u_0(\widehat x_\epsilon)
\end{equation}
Then, 
$$(|u_\epsilon'|^{p-2}u_\epsilon')(x)-(|u_0'|^{p-2}u_0')(x)-\left((|u_\epsilon'|^{p-2}u_\epsilon')(\widehat x_\epsilon)-(|u_0'|^{p-2}u_0')(\widehat x_\epsilon)\right)$$
$$= \int_{\widehat x_\epsilon}^x(v_\epsilon-v_0) \nearrow \epsilon-\left((|u_\epsilon'|^{p-2}u_\epsilon')(\widehat x_\epsilon)-(|u_0'|^{p-2}u_0')(\widehat x_\epsilon)\right) \quad \hbox{as } x\nearrow \ell.$$
Therefore,
$$(|u_\epsilon'|^{p-2}u_\epsilon')(x)-(|u_0'|^{p-2}u_0')(x)\le  \epsilon \quad \hbox{for } \widehat x_\epsilon<x<\ell.$$
From here, and using~\eqref{mue02001},
\begin{equation}\label{simcas01001}
0\le u_\epsilon(\ell)-u_0(\ell)\le \int_{\widehat x_\epsilon}^\ell\varrho^{-1}\left(\varrho(u_0'(x))+\epsilon\right)+u_0(\widehat x_\epsilon)-u_0(\ell).
\end{equation}

\noindent (b) If 
  $0<x_\epsilon<\ell$, then we have that
$$(|u_\epsilon'|^{p-2}u_\epsilon')(x_\epsilon)=(|u_0'|^{p-2}u_0')(x_\epsilon).$$
So, if
we apply the first point proved on this interval $(x_\epsilon,\ell)$,  we get that $u_0(x)\le u_\epsilon(x)$ in such interval and $u_0(\ell)< u_\epsilon(\ell)$; and we get that $u_0(x)\ge u_\epsilon(x)$ on the interval $(0,x_\epsilon)$ and $u_0(0)> u_\epsilon(0)$.
Observe that at point $x_\epsilon$, by continuity,
\begin{equation}\label{tkae001} u_\epsilon(x_\epsilon)=u_0(x_\epsilon).
\end{equation}
 Moreover, the proof of  the first point gives us the following fact  (using~\eqref{tkae001}): 
\begin{equation}\label{simcas01002}0\le u_\epsilon(\ell)-u_0(\ell)\le \int_{x_\epsilon}^\ell\varrho^{-1}\left(\varrho(u_0'(x))+\epsilon\right)+u_0(x_\epsilon)-u_0(\ell).
\end{equation}

 Observe that in both cases, (a) or (b), we obtain the similar inequalities~\eqref{simcas01001} and~\eqref{simcas01002}. Let us rewrite $x_\epsilon$ as $\widehat x_\epsilon$ in~\eqref{simcas01002}.

Then, like in the proof of the first point,  we get that
$$\lim_{\epsilon\to 0^+}u_\epsilon(\ell)=u_0(\ell).$$
 Indeed, we have that there exists $\widehat x_{\epsilon_n}\to x_0\in[0,\ell]$ for a sequence $\epsilon_n\to 0^+$. 
 Now,
$$
\int_{\widehat x_{\epsilon_n}}^\ell\varrho^{-1}\left(\varrho(u_0'(x))+\epsilon_n\right)=\int_{\widehat x_{\epsilon_n}}^{x_0}\varrho^{-1}\left(\varrho(u_0'(x))+\epsilon_n\right)+\int_{x_0}^\ell\varrho^{-1}\left(\varrho(u_0'(x))+\epsilon_n\right).
$$
Here, the fist integral on the right hand-side converges to $0$ as $\epsilon_n$ goes to $0$ since  the function $\varrho^{-1}\left(\varrho(u_0'(x))+\epsilon_n\right)$ is bounded uniformly in $n$, and, using this boundedness again,  the Dominate Convergence Theorem gives that the second integral on the right hand-side converges to $u_0(\ell)-u_0(x_0)$.  Then, we have
$$u_{\epsilon_n}(\ell)\to u_0(\ell)\quad\hbox{as }\epsilon_n\to 0,$$ 
and consequently, by the monotonicity,
$$\lim_{\epsilon\to 0^+}u_\epsilon(\ell)=u_0(\ell).$$

Similarly we get $\lim_{\epsilon\to 0^-}u_\epsilon(\ell)=u_0(\ell).$ 
Hence we have the  continuity of $\epsilon\mapsto u_\epsilon(\ell)$  at~$0$, and then at any point.
\end{proof}

\begin{remark}\label{030326_001rem}\rm  
 Observe that from the above lemma we have that if $v_{\delta,\epsilon}=\gamma(u_{\delta,\epsilon})$ is weak solution to problem ($P_{p,a-\delta,b+\epsilon}^{\gamma,f}$),  with $0\le\delta<\epsilon$ or $0<\delta\le\epsilon$,  then
 $$u_{\delta,\epsilon}(\ell)> u_{0,0}(\ell),$$
 Indeed, from Point 3 in the above lemma  we have that
 $$u_{\delta,\epsilon}(\ell)\ge u_{\delta,\delta}(\ell),$$
 which is an strict inequality if $\epsilon>\delta$,
 and from Point 1 we have that
 $$u_{\delta,\delta}(\ell)\ge u_{0,0}(\ell),$$
 which is an strict inequality if $\delta>0.$ 
 \hfill $\blacksquare$
\end{remark}

\begin{lemma}\label{lemma5}  Let  $\G$ be a  connected and compact metric graph with $$\V(\G) = \{ \v_1, \ldots, \v_{n-1},\v_n \}.$$
Assume $\gamma_\e:\mathbb{R}\to \mathbb{R}$ is a  continuous and increasing function with 
$\gamma_\e(\mathbb{R})=\mathbb{R}$ and $\gamma_\e(0)=0$ for all $\e \in \E(\G)$, and, for $\omega=\{\omega_{\v_1}, \ldots, \omega_{\v_{n-1}},\omega_{\v_n}  \} \subset \R$ and $\epsilon\in \R$,  set $\omega^\epsilon$ be such that $$\omega^\epsilon_{\v_i}=\omega_{\v_i}, \ i=1,2,...,n-1,$$ 
$$\omega^\epsilon_{\v_n}=\omega_{\v_n}+\epsilon.$$ 
Let $v_\epsilon = \overline{\gamma}(u_\epsilon)$ be weak solution  to problem $(P_{\omega^\epsilon}^f)$,   $\epsilon\in \mathbb{R}$. 
 Then, 
 $$u_0 \leq u_\epsilon \quad \hbox{if} \quad  \epsilon \ge 0, \quad \hbox{and} \quad  u_\epsilon \leq u_0 \quad \hbox{if} \quad \epsilon \le 0.$$
Moreover
 $$ \epsilon\mapsto u_\epsilon(\v_n) \ \hbox{ is increasing and continuous.}$$ 
\end{lemma}

\begin{proof}     
Let $\E_{\v_n}(\G)=\{\e_i:i=1,2,...,k\}$, and suppose that $\v_n=\vf_{\e_i}$ for all $i\in\{1,2,...,k\}$. Let us call $[v_\epsilon]_{\e_i}=v_{\epsilon,i}$ and $[u_\epsilon]_{\e_i}=u_{\epsilon,i}$, $\ell_{\e_i}=\ell_i$,  and $p_{\e_i}=p_i$, for each $i\in\{1,2,...,k\}$. 

Take $\epsilon>0$. By Theorem~\ref{teocontprin001} we have that $$v_\epsilon\ge v_0,$$ and consequently we have 
$$u_\epsilon\ge u_0,$$ which proves the first statement given.
In particular we have that
\begin{equation}\label{nmced001}
   0\le \lim_{\epsilon\to 0^+}\left(u_\epsilon(\v_n)-u_0(\v_n)\right).
\end{equation}
We want to show that this limit is $0$, and that $u_\epsilon(\v_n)> u_0(\v_n).$

By the mass balance property we have that 
\begin{equation}\label{conv0001}
  0\le\int_0^{\ell_i} (v_{\epsilon,i} - v_{0,i})\le
 \sum_{i=1}^k\int_0^{\ell_i} (v_{\epsilon,i} - v_{0,i})\le \epsilon,\quad i=1,2,...,k.
\end{equation}

By Proposition~\ref{lemma2} we have that, for each $i\in\{1,2,...,k\}$,
\begin{equation}\label{conv0002conx}
\begin{array}{c}
\displaystyle
(|u_{\epsilon,i}'|^{p_i-2}u_{\epsilon,i}') (x) -  (|u_{0,i}'|^{p_i-2}u_{0,i}') (x)   
\\ \\
\displaystyle=  \int_y^{x}
(v_{\epsilon,i} - v_{0,i})+  (|u_{\epsilon,i}'|^{p_i-2}u_{\epsilon,i}') (y) -  (|u_{0,i}'|^{p_i-2}u_{0,i}') (y)\quad\forall 0\le y\le x\le\ell_i.
\end{array}
\end{equation}

By the hypothesis $\omega^\epsilon_{\v_n}=\omega_{\v_n}+\epsilon$, then we have that
\begin{equation}\label{100326001}
\sum_{i=1}^k(|u_{\epsilon,i}'|^{p_i-2}u_{\epsilon,i}') (\ell_i) -  (|u_{0,i}'|^{p_i-2}u_{0,i}') (\ell_i)=\epsilon,
\end{equation}
Hence, there exists an index $i$ such that  $(|u_{\epsilon,i}'|^{p_i-2}u_{\epsilon,i}') (\ell_i) -  (|u_{0,i}'|^{p_i-2}u_{0,i}')(\ell_i) >0$.   Then, arguing as in the proof of Lemma~\ref{lemma3}, since $u_\epsilon\ge u_0$, we get that
$$      u_\epsilon(\v_n)> u_0(\v_n).$$
The growth for $\epsilon<0$ is proved similarly.
 
Let us see that there there is a positive sequence $\epsilon_{m}\to 0$ such that 
 $$\lim_{j}\left(u_{\epsilon_{m}}(\v_n)-u_0(\v_n)\right)=0.$$
By  the monotonicity, this is enough to prove $\lim_{\epsilon\to 0^+}\left(u_\epsilon(\v_n)-u_0(\v_n)\right)=0$.

Let $\epsilon>0$ be. Since, for each $i\in\{1,2,...,k\}$, $\displaystyle 0\le\int_0^{\ell_i} (v_{\epsilon,i} - v_{0,i})\le \epsilon$, we have that there exists  a positive sequence $  \epsilon_m\to 0$ such that  $$\lim_m v_{\epsilon_m,i}(x)= v_{0,i}(x)\quad\hbox{for a.e. } x\in (0,\ell_i),\ \hbox{for all } i\in\{1,2,...,k\},$$
and also
$$\lim_m u_{\epsilon_m,i}(x)= u_{0,i}(x)\quad\hbox{for a.e. } x\in (0,\ell_i), \ \hbox{for all } i\in\{1,2,...,k\}.$$
For each $i\in\{1,2,...,k\}$, fix $x_{0,i}\in (0,\ell_i)$ such that 
$$\lim_m u_{\epsilon_m,i}(x_{0,i})= u_{0,i}(x_{0,i}).$$

From~\eqref{conv0002conx}, in particular, 
\begin{equation}\label{conv0002}
\begin{array}{c}
\displaystyle
(|u_{\epsilon,i}'|^{p_i-2}u_{\epsilon,i}') (\ell_i) -  (|u_{0,i}'|^{p_i-2}u_{0,i}') (\ell_i)   
\\ \\
\displaystyle= \int_0^{\ell_i} (v_{\epsilon,i} - v_{0,i})+  (|u_{\epsilon,i}'|^{p_i-2}u_{\epsilon,i}') (0) -  (|u_{0,i}'|^{p_i-2}u_{0,i}') (0).
\end{array}
\end{equation}
Then, adding these expressions from $i=1$ to $i=k$,  since $\omega^\epsilon_{\v_n}=\omega_{\v_n}+\epsilon$, we have
$$\epsilon   
=\sum_{i=1}^k\int_0^{\ell_i} (v_{\epsilon,i} - v_{0,i})+ \sum_{i=1}^k\left( (|u_{\epsilon,i}'|^{p_i-2}u_{\epsilon,i}') (0) -  (|u_{0,i}'|^{p_i-2}u_{0,i}') (0)\right).$$
Then, by~\eqref{conv0001},  
\begin{equation}\label{conv0005}
\sum_{i=1}^k\left( (|u_{\epsilon,i}'|^{p_i-2}u_{\epsilon,i}') (0) -  (|u_{0,i}'|^{p_i-2}u_{0,i}') (0)\right)\ge 0.
\end{equation}

Write  $\epsilon=\epsilon_m$  for  simplicity. On account of~\eqref{conv0005}, we have two possible situations,  (i) or~(ii):

  \noindent  (i)     there is an $i\in\{1,2,...,k\}$ such that 
$$ (|u_{\epsilon,i}'|^{p_i-2}u_{\epsilon,i}') (0) -  (|u_{0,i}'|^{p_i-2}u_{0,i}') (0)< 0;$$

\noindent (ii) for all the  indices $i$,
$$ (|u_{\epsilon,i}'|^{p_i-2}u_{\epsilon,i}') (0) -  (|u_{0,i}'|^{p_i-2}u_{0,i}') (0)\ge  0.$$ 

In the case   (i),  by~\eqref{conv0002} and~\eqref{conv0001},
\begin{equation}\label{satis2701002}0\le (|u_{\epsilon,i}'|^{p_i-2}u_{\epsilon,i}') (\ell_i) -  (|u_{0,i}'|^{p_i-2}u_{0,i}') (\ell_i)\le  \epsilon,
\end{equation}
Now,
$$
0\le \int_{x_{0,i}}^x(v_{\epsilon,i}-v_{0,i})\nearrow \int_{x_{0,i}}^{\ell_i}(v_{\epsilon,i}-v_{0,i}) \quad\hbox{as }x\nearrow \ell_i,
$$
hence,
$$
(|u_{\epsilon,i}'|^{p_i-2}u_{\epsilon,i}') (x) -  (|u_{0,i}'|^{p_i-2}u_{0,i}')(x) \le (|u_{\epsilon,i}'|^{p_i-2}u_{\epsilon,i}') (\ell_i) -  (|u_{0,i}'|^{p_i-2}u_{0,i}')(\ell_i)\le\epsilon 
$$
for  all $x_{0,i}<x<\ell_i.$
Then, using $\varrho(r):= \vert r \vert^{p-2}r$,
 $$u_{\epsilon,i}'(x)\le \varrho^{-1}\left(\varrho(u_{0,i}'(x))+\epsilon\right)\quad \ \forall x_{0,i}<x<\ell_i.$$
Integrating in the above expression,
 $$u_{\epsilon,i}(\ell_i)-u_{\epsilon,i}(x_{0,i})\le \int_{x_{0,i}}^{\ell_i}\varrho^{-1}\left(\varrho(u_{0,i}'(x))+\epsilon\right).$$
Therefore, 
\begin{equation}\label{27f22201}
0\le u_{\epsilon}(\v_n)-u_{0}(\v_n)\le \int_{x_{0,i}}^{\ell_i}\varrho^{-1}\left(\varrho(u_{0,i}'(x))+\epsilon\right)+u_{\epsilon,i}(x_{0,i})-u_{0,i}(\ell_i).
\end{equation}

In the case (ii)  we also have that, for all $i$,
$$ (|u_{\epsilon,i}'|^{p_i-2}u_{\epsilon,i}') (\ell_i) -  (|u_{0,i}'|^{p_i-2}u_{0,i}') (\ell_i)\ge  0,$$
otherwise, we get a contradiction with the growth of the solution (Lemma \ref{lemma3}).  Hence, since $\omega^\epsilon_{\v_n}=\omega_{\v_n}+\epsilon,$
\begin{equation}\label{satis2701001} 0\le (|u_{\epsilon,i}'|^{p_i-2}u_{\epsilon,i}') (\ell_i) -  (|u_{0,i}'|^{p_i-2}u_{0,i}') (\ell_i)\le  \epsilon.
\end{equation}
Then for any $i$,  following the same steps that in case (i) we get,  
$$0\le u_{\epsilon}(\v_n)-u_{0}(\v_n)\le \int_{x_{0,i}}^{\ell_i}\varrho^{-1}\left(\varrho(u_{0,i}'(x))+\epsilon\right)+u_{\epsilon,i}(x_{0,i})-u_{0,i}(\ell_i).$$

Therefore, in any situation, since we have finite indices in $\{1,2,...,k\}$, there exists a subsequence $\epsilon_{m}$, that we denote equal, such that, for an unique $i\in \{1,2,...,k\}$,
$$0\le u_{\epsilon_{m}}(\v_n)-u_{0}(\v_n)\le \int_{x_{0,i}}^{\ell_i}\varrho^{-1}\left(\varrho(u_{0,i}'(x))+\epsilon_{m}\right)+u_{\epsilon_{m},i}(x_{0,i})-u_{0,i}(\ell_i),$$
and $$u_{\epsilon_{m},i}(x_{0,i})\to u_{0,i}(x_{0,i}).$$
Now we finish as in the proof of the previous lemma.
\end{proof}

\begin{lemma}\label{antepenultimo}
Let  $\G$ be a  connected and compact metric graph with $$\V(\G) = \{ \v_1, \v_2, \ldots, \v_m \}.$$
Assume $\gamma_\e:\mathbb{R}\to \mathbb{R}$ is a  continuous and increasing function with 
$\gamma_\e(\mathbb{R})=\mathbb{R}$ and $\gamma_\e(0)=0$ for all $\e \in \E(\G)$ and 
$\omega=\{\omega_{\v_1},  \omega_{\v_{2}},\ldots,\omega_{\v_m}  \} \subset \R$. Let $j$ and $k$ be any natural numbers  such that $j+k\le m$. Take $\epsilon_i> 0$ for $i=1,2,\ldots,j$ and $i=j+1,j+2,\ldots,j+k$.
 Let $\omega^\epsilon$ be
such that 
$$\omega^\epsilon_{\v_i}=\omega_{\v_i}-\epsilon_i,\ i=1,2,\ldots,j$$
$$\omega^\epsilon_{\v_i}=\omega_{\v_i}+\epsilon_i,\ i=j+1,j+2,\ldots,j+k$$
$$\omega^\epsilon_{\v_i}=\omega_{\v_i}\ \hbox{in other  vertices.}$$ 
Let $v_\epsilon = \overline{\gamma}(u_\epsilon)$ be weak solution  to problem $(P_{\omega^\epsilon}^f)$, and $v_0 = \overline{\gamma}(u_0)$ be weak solution  to problem $(P_{\omega}^f)$. 
 Then, 
\item[(i)] If   $\displaystyle\sum_{i=1}^j\epsilon_i\le \sum_{i=j+1}^{j+k}\epsilon_i,$  we have 
 $ u_0(\v_i) \le u_\epsilon(\v_i)  \ \hbox{ for $i=j+1,j+2,\ldots,j+k$.}$
 \item[(ii)] If   $\displaystyle\sum_{i=1}^j\epsilon_i\ge \sum_{i=j+1}^{j+k}\epsilon_i,$   we have 
 $ u_\epsilon(\v_i) \le u_0(\v_i)  \ \hbox{ for $i=1,2,\ldots,j$.}$
\end{lemma}

\begin{proof}  
Let us call to problem $(P_{\omega^\epsilon}^f)$ the perturbed problem.  As in the previous proof, let us call $[v_\epsilon]_{\e_i}=v_{\epsilon,i}$ and $[u_\epsilon]_{\e_i}=u_{\epsilon,i}$, $\ell_{\e_i}=\ell_i$,  and $p_{\e_i}=p_i$.

We will prove (i), (ii) is proved similarly. The proof is done by induction on the number $N$ of edges of the graph. The case $N=1$ has been proved in Lemma~\ref{lemma3}  (see Remark~\ref{030326_001rem}),  a simple situation with two edges is also proved in Remark~\ref{nuevolemma01}. Suppose that we have the result for any   connected  graph of $k$ edges with $1\le k\le N$, $N\in\mathbb{N},$ and let us prove the result for a graph with $N+1$ edges. So, take such a graph with $N+1$ edges and suppose it has $m$ vertices, and that at $j$ vertices   the flux  in  the perturbed problem decreases with respect to the one in the non perturbed problem  and at $k$ vertices  the flux increases, take the same  notation of the   statement. For the  $k$ vertices where the flux increases we assume that they are final vertices in their parametrization.

By the hypothesis $\omega^\epsilon_{\v_{j+1}}=\omega_{\v_{j+1}}+\epsilon_{j+1}$, $\epsilon_{j+1}>0$,  we have that  
\begin{equation}\label{100326001nl}
\sum_{\e\in\E_{\v_{j+1}}(\G)}(|u_{\epsilon,\e}'|^{p_\e-2}u_{\epsilon,\e}') (\ell_\e) -  (|u_{0,\e}'|^{p_\e-2}u_{0,\e}') (\ell_\e)=\epsilon_{j+1}>0.
\end{equation}
   Hence the set $\E_{\v_{j+1}}(\G)^+\subset \E_{\v_{j+1}}(\G)$ of edges $\e$ such that
\begin{equation}\label{130326001} (|u_{\epsilon,\e}'|^{p_{\e}-2}u_{\epsilon,\e}') (\ell_{\e}) -  
  (|u_{0,\e}'|^{p_{\e}-2}u_{0,\e}')(\ell_{\e})>0,
\end{equation}
is not empty. Set 
$m_{j+1}=\sharp(\E_{\v_{j+1}}(\G))$.

Consider now the graph $\widetilde\G$ obtained from $\G$ by splitting the vertex $\v_{j+1}$ into $m_{j+1}$ vertices  so that each of them is the terminal vertex of exactly one edge in $\E_{\v_{j+1}}(\G)$, let us call these vertices
$$\v_{j+1,i},\quad i=1,2,\ldots, m_{j+1}.$$
All the other vertices and edges are preserved as in $\G$. 
See Figure~\ref{fig:vertex-splitting-example} (the splitting vertex is~$\v_7$). And on this new graph we consider fluxes on the vertices according to the decomposition obtained from Proposition~\ref{lemma2} for the solution with $\omega$ and for the one of the perturbed problem with $\omega^\epsilon$.    Since we are using the same edges and fluxes, the solutions we had for $\G$ are also solutions for this new graph   $\widetilde\G$.
 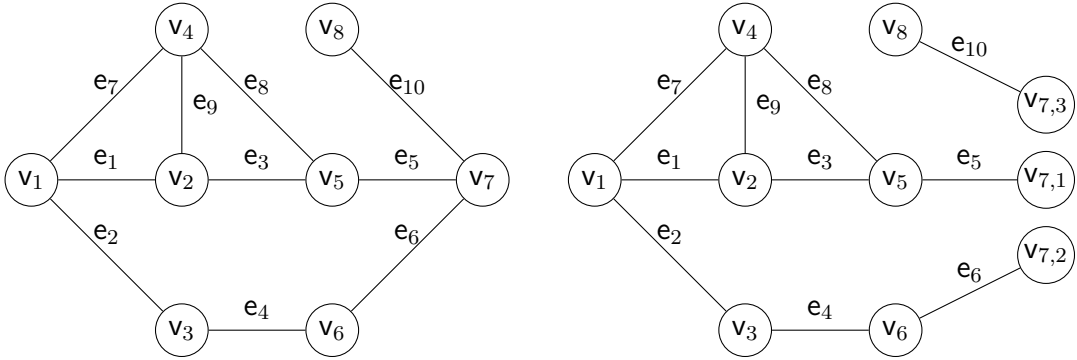
\begin{figure}[ht]
\centering

\begin{tikzpicture}[
vertex/.style={circle, draw, minimum size=7mm, inner sep=1pt},
node distance=2cm
]

\node[vertex] (v1) {$\v_1$};
\node[vertex, right of=v1] (v2) {$\v_2$};
\node[vertex, below of=v2] (v3) {$\v_3$};
\node[vertex, above of=v2] (v4) {$\v_4$};
\node[vertex, right of=v2] (v5) {$\v_5$};
\node[vertex, right of=v3] (v6) {$\v_6$};
\node[vertex, right of=v5] (v7) {$\v_7$};
\node[vertex, above of=v5] (v8) {$\v_8$};

\draw (v1) -- node[midway, above] {$\e_1$} (v2);
\draw (v1) -- node[midway, above] {$\e_2$} (v3);
\draw (v2) -- node[midway, above] {$\e_3$} (v5);
\draw (v3) -- node[midway, above] {$\e_4$} (v6);
\draw (v5) -- node[midway, above] {$\e_5$} (v7);
\draw (v6) -- node[midway, above] {$\e_6$} (v7);
\draw (v1) -- node[midway, above] {$\e_7$} (v4);
\draw (v4) -- node[midway, above] {$\e_8$} (v5);
\draw (v4) -- node[midway, right] {$\e_9$} (v2);
\draw (v7) -- node[midway, above] {$\e_{10}$} (v8);

\end{tikzpicture}
\hspace{0.5cm}
\begin{tikzpicture}[
vertex/.style={circle, draw, minimum size=7mm, inner sep=1pt},
node distance=2cm
]

\node[vertex] (v1) {$\v_1$};
\node[vertex, right of=v1] (v2) {$\v_2$};
\node[vertex, below of=v2] (v3) {$\v_3$};
\node[vertex, above of=v2] (v4) {$\v_4$};
\node[vertex, right of=v2] (v5) {$\v_5$};
\node[vertex, right of=v3] (v6) {$\v_6$};

\node[coordinate] (v7pos) [right of=v5] {};

\node[vertex] (v71) at ($(v7pos)$) {$\v_{7,1}$};
\node[vertex] (v72) at ($(v7pos)+(0,-1.0)$) {$\v_{7,2}$};
\node[vertex] (v73) at ($(v7pos)+(0,1.0)$) {$\v_{7,3}$};
\node[vertex, above of=v5] (v8)   {$\v_8$};

\draw (v1) -- node[midway, above] {$\e_1$} (v2);
\draw (v1) -- node[midway, above] {$\e_2$} (v3);
\draw (v2) -- node[midway, above] {$\e_3$} (v5);
\draw (v3) -- node[midway, above] {$\e_4$} (v6);

\draw (v5) -- node[midway, above] {$\e_5$} (v71);
\draw (v6) -- node[midway, above] {$\e_6$} (v72);

\draw (v1) -- node[midway, above] {$\e_7$} (v4);
\draw (v4) -- node[midway, above] {$\e_8$} (v5);
\draw (v4) -- node[midway, right] {$\e_9$} (v2);

\draw (v73) -- node[midway, above] {$\e_{10}$} (v8);

\end{tikzpicture}

\caption{A graph \(\G\)  and the graph obtained from it by splitting the vertex \(\v_7\).}
\label{fig:vertex-splitting-example}
\end{figure}

It could appear a  subgraph $\widehat\G$ of $\widetilde\G$ which is not connected to any of the vertices $\{\v_1,\v_2,\ldots,\v_j\}$, suppose that this happens with  vertex  $\v_{j+1,i_0}$  ($\v_{7,3}$ in Figure~\ref{fig:vertex-splitting-example}). Now, if in the decomposition for the non perturbed problem and the perturbed one via Proposition~\ref{lemma2},  the   flux in  $\v_{j+1,i_0}$  for the perturbed problem is smaller than the one for the non perturbed, 
by Lemma~\ref{lemma5}, we have
that the solution in such vertex   decreases respect to the solution $v_0=\gamma(u_0)$, so the same happens at vertex  $\v_{j+1}$. But, if we delete the associated subgraph $\widehat\G$ from $\G$, maintaining the rest of edges and vertices and the fluxes determined via  Proposition~\ref{lemma2} for the perturbed problem and the non perturbed one, this graph is under the induction hypothesis and we get a contradiction with the continuity of the solution since we will also get that the solution at vertex $\v_{j+1}$   increases respect to the solution $v_0=\gamma(u_0)$. Then, if such a disconnected subgraph $\widehat\G$ appears, then the  difference of flux on the vertex  $\v_{j+1,i_0}$ is null or positive. If it is positive then,  by Lemma~\ref{lemma5}, we have
that the solution for the perturbed problem increases  in such vertex  respect to the solution $v_0=\gamma(u_0)$, and we have the growing property for $\v_{j+1}$. It it is null, and remove the subgraph, we can apply induction again on the remaining subgraph and we also get the growing property desired.

If the above disconnectedness occurs at all $\v_{j+i}$, $i=1,2,...,k,$ the proof is done. But if there is one for which it does not hold, we proceed as follows. Suppose that it does no occurs ad $\v_{j+1,1}$. Take $\e\in \E_{\v_{j+1}}(\G)^+$. Arguing as in the proof of Lemma~\ref{lemma3}, Point 3, either there exists $x_\epsilon\in [0,\ell_{\e}[$ such that
\begin{equation}\label{c12703}\begin{array}{c}\displaystyle u_{\epsilon,\e}(x)-u_{0,\e}(x)< u_{\epsilon,\e}(y)-u_{0,\e}(y)\quad\forall x_\epsilon< x<y\le\ell_{\e} \ \hbox{ and} \\ \\
\displaystyle(|u_{\epsilon,\e}'|^{p_{\e}-2}u_{\epsilon,\e}')(x_\epsilon)=(|u_{0,\e}'|^{p_{\e}-2}u_{0,\e}')(x_\epsilon),
\end{array}
\end{equation}
or 
\begin{equation}\label{c22703}\begin{array}{c}\displaystyle u_{\epsilon,\e}(x)-u_{0,\e}(x)< u_{\epsilon,\e}(y)-u_{0,\e}(y)\quad\forall 0\le x<y\le\ell_{\e} \ \hbox{ and} \\ \\
\displaystyle(|u_{\epsilon,\e}'|^{p_{\e}-2}u_{\epsilon,\e}')(0)-(|u_{0,\e}'|^{p_{\e}-2}u_{0,\e}')(0)>0.
\end{array}
\end{equation}
  In the first case~\eqref{c12703}, like in such proof, we get 
$$u_0(\v_{j+1,1}) \le u_\epsilon(\v_{j+1,1}),$$
which gives what we want to prove for this vertex. Consider now that we are in the second case~\eqref{c22703}. Then in the initial vertex of edge
$\e$ we have that
$$(|u_{\epsilon,\e}'|^{p_{\e}-2}u_{\epsilon,\e}')(0)-(|u_{0,\e}'|^{p_{\e}-2}u_{0,\e}')(0)>0.$$
Now if 
$$(|u_{\epsilon,\e}'|^{p_{\e}-2}u_{\epsilon,\e}')(0)-(|u_{0,\e}'|^{p_{\e}-2}u_{0,\e}')(0)\le (|u_{\epsilon,\e}'|^{p_{\e}-2}u_{\epsilon,\e}') (\ell_{\e}) -  
  (|u_{0,\e}'|^{p_{\e}-2}u_{0,\e}')(\ell_{\e}),$$
then, by Lemma~\ref{lemma3}, we get the desired growth condition. In other case, following the idea used previously, we
first obtain that at $\vi_\e$ the solution of the perturbed problem decreases respect to the one of the non perturbed problem by using only edge $\e$, and if we remove~$\e$, on the remaining subgraph we are under the induction hypothesis and we just get the contrary growing property at $\vi_\e$, which gives a contradiction with continuity. 
\end{proof}

  As a particular case  of the above lemma we have the following result. 

\begin{lemma}\label{030326_002}   Let  $\G$ be a  connected and compact metric graph with $$\V(\G) = \{ \v_1, \ldots, \v_{n-1},\v_n \}.$$
Assume $\gamma_\e:\mathbb{R}\to \mathbb{R}$ is a  continuous and increasing function with 
$\gamma_\e(\mathbb{R})=\mathbb{R}$ and $\gamma_\e(0)=0$ for all $\e \in \E(\G)$, and, for $\omega=\{\omega_{\v_1}, \ldots, \omega_{\v_{n-1}},\omega_{\v_n}  \} \subset \R$ and $\epsilon\in \R$,  set $\omega^\epsilon$ be such that $$\omega^\epsilon_{\v_i}=\omega_{\v_i}, \ i=2,...,n-2,$$ 
$$\omega^\epsilon_{\v_n}=\omega_{\v_n}+\epsilon,$$
$$\omega^\epsilon_{\v_1}=\omega_{\v_1}-\epsilon.$$
  Let $v_\epsilon = \overline{\gamma}(u_\epsilon)$ be weak solution  to problem $(P_{\omega^\epsilon}^f)$, also for $\epsilon=0$.
 Then, 
 $$u _0(\v_n) \le u_\epsilon(\v_n) \ \hbox{ and }\   u_\epsilon(\v_1)\leq  u_0(\v_1)  \quad \hbox{if } \   \epsilon \ge 0,
 $$
 and 
 $$u_\epsilon(\v_n)\le u_0(\v_n)    \ \hbox{ and }\  u_\epsilon(\v_1)\ge u_0(\v_1)  \quad \hbox{if } \   \epsilon \le 0.$$
 \end{lemma}

 The following remark illustrates in a simple case which is the idea for the proof of the next lemma.
\begin{remark}\label{nuevolemma01}\rm
   Let $\G$ be a graph with two edges $\{\e_1,\e_2\}$ like in Figure~\ref{fig:002}, $\vi_{\e_1}=\v_1, \vf_{\e_1}=\v_2=\vi_{\e_2}, \vf_{\e_2}=\v_3$. Let, for $i=1,2$, $\gamma_i:=\gamma_{\e_i}:\mathbb{R}\to \mathbb{R}$ 
be a continuous and increasing function with $\gamma_i(\mathbb{R})=\mathbb{R}$ and $\gamma_i(0)=0$, and $p_i:=p_{\e_i}\in(1,+\infty)$.
Let $\epsilon\ge 0$. Put the following fluxes on vertices:
$$\omega_{\v_1}^\epsilon=a-\epsilon, \ \omega_{\v_2}^\epsilon=b,\, \omega_{\v_3}^\epsilon=c+\epsilon,$$
call $\omega^\epsilon=\{\omega^\epsilon_{\v_1},\omega^\epsilon_{\v_2},\omega^\epsilon_{\v_3}\}$.
\begin{figure}[ht]
\centering
\begin{tikzpicture}[scale=1,
  midarrow/.style={
    postaction={decorate},
    decoration={
      markings,
      mark=at position 0.5 with {\arrow{>}}
    }
  }
]

\node[circle, draw] (v1) at (0,0) {$\v_1$};
\node[circle, draw] (v2) at (2.2,0) {$\v_2$};
\node[circle, draw] (v3) at (5.2,0) {$\v_3$};

\draw[midarrow] (v1) -- node[above] {$\e_1$} (v2);
\draw[midarrow] (v2) -- node[above] {$\e_2$} (v3);

\end{tikzpicture}
\caption{Arrows indicate the parametrization used.}
\label{fig:002}
\end{figure}
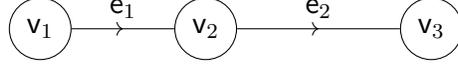
Let $v_\epsilon=\overline\gamma(u_\epsilon)$ be  solution of $(P_{\omega^\epsilon}^f)$. 
Call $v_{\epsilon,i}=[v_\epsilon]_{\e_i}$, $u_{\epsilon,i}=[u_\epsilon]_{\e_i}$, $\ell_{\e_i}=\ell_i$. We have that $\epsilon\mapsto u_{\epsilon,2}(\ell_2)$ is continuous and increasing in $[0,+\infty)$,  $\epsilon\mapsto u_{\epsilon,1}(0)$ is continuous and decreasing in $[0,+\infty)$,  and 
\begin{equation}\label{luns0001} \lim_{\epsilon\to +\infty}u_{\epsilon,2}(\ell_2)=+\infty 
 \ \hbox{ or }\ \lim_{\epsilon\to +\infty}u_{\epsilon,1}(0)=-\infty.
\end{equation}

Indeed, by Proposition~\ref{lemma2}, we have that
$$v_{\epsilon,i}=\gamma_{i}(u_{\epsilon,i}) \ \hbox{is solution of problem} \ P^{\gamma_i, [f]_{\e_1}}_{ p_i, a_{\e_i}^\epsilon,b_{\e_i}^\epsilon} , \ i= 1,2,$$  for adequate $a_{\e_i}^\epsilon,b_{\e_i}^\epsilon$,  that,   looking at the vertex $\v_2$, are
$$  a_{\e_2}^\epsilon= -(|u_{\epsilon,2}'|^{p_2-2}u_{\epsilon,2}')(0), \quad  b_{\e_1}^\epsilon= (|u_{\epsilon,1}'|^{p_1-2}u_{\epsilon,1}')(\ell_{1}),$$ with 
$$b=a_{\e_2}^\epsilon+b_{\e_1}^\epsilon.$$
 Then, we have
$$a_{\e_2}^\epsilon= -(|u_{0,2}'|^{p_{2}-2}u_{0,2}')(0) -\Big((|u_{\epsilon,2}'|^{p_{2}-2}u_{\epsilon,2}')(0)
 -(|u_{0,2}'|^{p_{2}-2}u_{0,2}')(0) \Big),$$
$$b_{\e_1}^\epsilon= (|u_{0,1}'|^{p_{1}-2}u_{0,1}')(\ell_{1}) +\Big((|u_{\epsilon,1}'|^{p_{1}-2}u_{\epsilon,1}')(\ell_{1})- (|u_{0,1}'|^{p_{1}-2}u_{0,1}')(\ell_{1}) \Big).$$
Observe that, since $\omega^\epsilon_{\v_2}=b$,
$$(|u_{\epsilon,2}'|^{p_{2}-2}u_{\epsilon,2}')(0)
 -(|u_{0,2}'|^{p_{2}-2}u_{0,2}')(0) =(|u_{\epsilon,1}'|^{p_{1}-2}u_{\epsilon,1}')(\ell_{1})- (|u_{0,1}'|^{p_{1}-2}u_{0,1}')(\ell_{1})$$ let us call it $\sigma_\epsilon$.
Hence  
$$a_{\e_2}^\epsilon=a_{\e_2}^0-\sigma_\epsilon,\ b_{\e_1}^\epsilon=b_{\e_1}^0+\sigma_\epsilon.$$

If $$\sigma_\epsilon >\epsilon,$$ on the one hand we have
$$(|u_{\epsilon,2}'|^{p_{2}-2}u_{\epsilon,2}')(0)
 -(|u_{0,2}'|^{p_{2}-2}u_{0,2}')(0) > \epsilon,$$
 which by Lemma~\ref{lemma3} implies that
 $$ u_{\epsilon,2}(\v_2) <u_{0,2}(\v_2).$$
 On the other hand
 $$(|u_{\epsilon,1}'|^{p_{1}-2}u_{\epsilon,1}')(\ell_{1})- (|u_{0,1}'|^{p_{1}-2}u_{0,1}')(\ell_{1}) > \epsilon,$$
  which, by Lemma~\ref{lemma3}, implies that
$$ u_{\epsilon,1}(\v_2) > u_{0,1}(\v_2).$$
Now these two fact  break down the continuity at $\v_2$. Consequently, we have $$\sigma_\epsilon\le\epsilon.$$
and, then, from our previous result Lemma~\ref{lemma3} we get the growth property  (observe that this idea is developed in the proof of Lemma~\ref{antepenultimo} for a general case).

Let us now see that $\displaystyle\lim_{\epsilon\to 0^+}u_{\epsilon,2}(\ell_2)=u_{0,2}(\ell_2)$.
Indeed, by Lemma~\ref{lemma3},  we have that $u_{\epsilon,2} \le \widetilde u_\epsilon$, where $ \widetilde u_\epsilon$ is the solution of $(P_{p_2,-\epsilon,\epsilon}^{\gamma_{\e_2},[f]_{\e_2}})$. Hence,
$$0\le u_{\epsilon,2}(\ell_2)-u_0(\ell_2)\le \widetilde u_\epsilon(\ell_2)-u_0(\ell_2),$$
and, by Lemma~\ref{lemma3}, $\displaystyle\lim_{\epsilon\to 0^+}\left( \widetilde u_\epsilon(\ell_2) -u_0(\ell_2)\right)=0.$ Hence, $$\displaystyle\lim_{\epsilon\to 0^+}u_{\epsilon,2}(\ell_2)=u_{0,2}(\ell_2).$$
Similarly, $\displaystyle\lim_{\epsilon\to 0^-}u_{\epsilon,2}(\ell_2)=u_{0,2}(\ell_2).$
And from here we get the continuity. Similarly for $\epsilon\mapsto u_{\epsilon,1}(0)$.

  Let us finally prove~\eqref{luns0001}. 
Take $\epsilon_n$ a sequence such that $\lim_n\epsilon_n=+\infty$. Now, there exist a subsequence that we denote equal,  such that (i) or (ii) happen, where:

\noindent (i) the two following properties hold:
$$(|u_{\epsilon_n,2}'|^{p_{2}-2}u_{\epsilon_n,2}')(x)
 -(|u_{0,2}'|^{p_{2}-2}u_{0,2}')(x)>0\quad\hbox{ for all } 0<x<\ell_2,$$
and
$$(|u_{\epsilon_n,1}'|^{p_{1}-2}u_{\epsilon_n,1}')(x)
 -(|u_{0,1}'|^{p_{1}-2}u_{0,1}')(x)>0\quad\hbox{ for all } 0<x<\ell_1,$$

\noindent (ii) there exists $x_{\epsilon_n}\in (0,\ell_2)$, for each $n$, such that $$(|u_{\epsilon_n,2}'|^{p_{2}-2}u_{\epsilon_n,2}')(x)
 -(|u_{0,2}'|^{p_{2}-2}u_{0,2}')(x)>0\quad\hbox{for all }x_{\epsilon_n}<x<\ell_2,$$
 and 
 $$(|u_{\epsilon_n,2}'|^{p_{2}-2}u_{\epsilon_n,2}')(x_{\epsilon_n})
 -(|u_{0,2}'|^{p_{2}-2}u_{0,2}')(x_{\epsilon_n})=0;$$
 or there exists $x_{\epsilon_n}\in (0,\ell_1)$, for each $n$, such that $$(|u_{\epsilon_n,1}'|^{p_{1}-2}u_{\epsilon_n,1}')(x)
 -(|u_{0,1}'|^{p_{1}-2}u_{0,1}')(x)>0\quad\hbox{for all }x_{\epsilon_n}<x<\ell_1,$$
 and 
 $$(|u_{\epsilon_n,1}'|^{p_{1}-2}u_{\epsilon_n,1}')(x_{\epsilon_n})
 -(|u_{0,1}'|^{p_{1}-2}u_{0,1}')(x_{\epsilon_n})=0.$$

In the case (i), we have that $$u_{\epsilon_n,2}'(x)\ge u_{0,2}'(x) \quad\hbox{ for all } 0<x<\ell_2,$$
and integrating,
\begin{equation}\label{2m26001}
u_{\epsilon_n,2}(x)-u_{0,2}(x)\ge u_{\epsilon_n,2}(y)-u_{0,2}(y)\quad\forall\, 0\le y\le x\le\ell_2.
\end{equation}
And $$u_{\epsilon_n,1}'(x)\ge u_{0,1}'(x) \quad\hbox{ for all } 0<x<\ell_1,$$
and integrating,
\begin{equation}\label{2m26001p1}
u_{\epsilon_n,1}(x)- u_{0,1}(x)\ge u_{\epsilon_n,1}(y)-u_{0,2}(y)\quad\forall\, 0\le y\le x\le\ell_1.
\end{equation}

Now,
$$\int_0^{\ell_2}(v_{\epsilon_n,2}-v_{0,2})=\epsilon_n-\sigma_{\epsilon_n}.$$
Hence if $$\lim_n(\epsilon_n-\sigma_{\epsilon_n})=+\infty,$$ by the previous equality, we get that there exists $ y_{n}\in (0,\ell_2)$ such that
$$\lim_n v_{\epsilon_n,2}(y_n)=+\infty,$$
and also,
$$\lim_n u_{\epsilon_n,2}(y_n)=+\infty.$$
Then, by~\eqref{2m26001},
$$\lim_n u_{\epsilon_n,2}(\ell_2)=+\infty.$$
Similarly, 
$$\lim_n u_{\epsilon_n,1}(0)=-\infty.$$

On the other hand, if 
$$\lim_n(\epsilon_n-\sigma_{\epsilon_n})\not=+\infty,$$
we have that $$\lim_n\sigma_{\epsilon_n}=+\infty.$$ 
Now, on account of~\eqref{2m26001}, \eqref{2m26001p1}  and continuity,
$$u_{\epsilon_n,1}(y)-u_{0,1}(y)\le u_{\epsilon_n,1}(\ell_1)-u_{0,1}(\ell_1)=u_{\epsilon_n,2}(0)-u_{0,2}(0)\le u_{\epsilon_n,2}(x)-u_{0,2}(x)$$  for $0\le y\le\ell_1,\ 0\le x\le \ell_2$.
So, if $$\alpha_n:=u_{\epsilon_n,1}(\ell_1)-u_{0,1}(\ell_1)=u_{\epsilon_n,2}(0)-u_{0,2}(0)\ge 0$$ then
$$u_{\epsilon_n,2}(x)\ge u_{0,2}(x)\quad\hbox{for all } 0\le x\le \ell_2,$$ and also
$$v_{\epsilon_n,2}(x)\ge v_{0,2}(x)\quad\hbox{for all } 0\le x\le \ell_2.$$
Hence,
$$0\le \int_{0}^{x}(v_{\epsilon_n,2}-v_{0,2})=(|u_{\epsilon_n,2}'|^{p_{2}-2}u_{\epsilon_n,2}')(x)-
(|u_{0,2}'|^{p_{2}-2}u_{0,2}')(x)
 -\sigma_{\epsilon_n}.$$
Then, using $\varrho(r):= \vert r \vert^{p-2}r$,
 $$u_{\epsilon_n,2}'(x)\ge\varrho^{-1}\left(\varrho(u_{0,2}'(x))+\sigma_{\epsilon_n}\right).$$
From here, integrating,
$$u_{\epsilon_n,2}(\ell_2)\ge u_{\epsilon_n,2}(0)+\int_{0}^{\ell_2}\varrho^{-1}\left(\varrho(u_{0,2}'(x))+\sigma_{\epsilon_n}\right).$$
 And from here, since $\sigma_{\epsilon_n}\to +\infty$ and $u_{\epsilon_n,2}(0)\ge u_{0,2}(0)$, we conclude that 
 $$\lim_n u_{\epsilon_n,2}(\ell_2)=+\infty.$$
 If $\alpha_n\le0$, arguing similarly with the edge $\e_1$, we get that
 $$\lim_n u_{\epsilon_n,1}(0)=-\infty.$$

In the case (ii),   suppose the case in which  there exists $x_{\epsilon_n}\in (0,\ell_2)$, for each $n$, such that
\begin{equation}\label{3m20002}(|u_{\epsilon_n,2}'|^{p_{2}-2}u_{\epsilon_n,2}')(x)
 -(|u_{0,2}'|^{p_{2}-2}u_{0,2}')(x)>0\quad\forall x_{\epsilon_n}<x<\ell_2
 \end{equation}
 and
$$(|u_{\epsilon_n,2}'|^{p_{2}-2}u_{\epsilon_n,2}')(x_{\epsilon_n})
 -(|u_{0,2}'|^{p_{2}-2}u_{0,2}')(x_{\epsilon_n})=0,$$ and let us see that  $\lim_n u_{\epsilon_n,2}(\ell_2)=+\infty$. The other case will give, with a similar argument, that $\lim_n u_{\epsilon_n,1}(0)=-\infty$.
 
 Then, in such case,
 $$\int_{x_{\epsilon_n}}^{\ell_2}(v_{\epsilon_n,2}-v_{0,2})=\epsilon,$$
 and therefore there exists $y_{\epsilon_n}\in (x_{\epsilon_n},\ell_2)$ such that 
 $$\lim_n v_{\epsilon_n,2}(y_n)=+\infty,$$
and also,
\begin{equation}\label{3m26004}\lim_n u_{\epsilon_n,2}(y_n)=+\infty.
\end{equation}
On the other hand, by~\eqref{3m20002}, as in the case (i),
\begin{equation}\label{3m26003}
u_{\epsilon_n,2}(x)\ge u_{\epsilon_n,2}(y)+u_{0,2}(x)-u_{0,2}(y)\quad\forall x_{\epsilon_n}\le y\le x\le\ell_2.
\end{equation}
Then, by this property and~\eqref{3m26004}, we get that
$$\lim_n u_{\epsilon_n,2}(\ell_2)=+\infty.$$ 
\hfill  $\blacksquare$
\end{remark}

\begin{lemma}\label{lemma51822} Let  $\G$ be a connected and compact metric graph   with vertices $$V(\G) = \{ \v_1,\v_2 \ldots, \v_{n-1},\v_n \},\quad n\ge 3.$$     Assume $\gamma_\e:\mathbb{R}\to \mathbb{R}$ is a  continuous and increasing function with $\gamma_\e(\mathbb{R})=\mathbb{R}$ and $\gamma_\e(0)=0$ for all $\e \in E(\G)$, and  assume $$\v_1,\v_n\in \partial \V(\G),$$ 
$$\v_1=\vi_{\e_1},\ \v_n=\vf_{\e_2},$$
($\e_1\neq \e_2$).
For $\omega=\{\omega_{\v_1},\omega_{\v_2}, \ldots, \omega_{\v_{n-1}},\omega_{\v_n}  \} \subset \R$ and $\epsilon\in \R$,  set $\omega^\epsilon$ be such that $$\omega^\epsilon_{\v_1}=\omega_{\v_1}-\epsilon,$$ $$ \omega^\epsilon_{\v_i}=\omega_{\v_i}, \ i=2,3,...,n-1,
$$ $$\omega^\epsilon_{\v_n}=\omega_{\v_n}+\epsilon.$$ Then, we have 
 
 \begin{itemize}
 \item[\rm  (i)] Let $v_\epsilon = \overline{\gamma}(u_\epsilon)$ be weak solution to problem $(P_{\omega^\epsilon}^f)$,  with $\epsilon \ge 0$. Then, $\epsilon\to u_\epsilon(\v_n)$ is  continuous and  increasing in $[0,+\infty)$, $\epsilon\to u_\epsilon(\v_1)$ is continuous and decreasing in $[0, +\infty)$, and
 \begin{equation}\label{1557dl} 
 \begin{array}{l}\displaystyle
 \lim_{\epsilon\to+\infty}u_\epsilon(\v_n)=+\infty 
 \hbox{ or  } 
 \lim_{\epsilon\to+\infty}u_\epsilon(\v_1)=-\infty.
 \end{array}
 \end{equation}
 
\item[\rm (ii)] Let $v_\epsilon = \overline{\gamma}(u_\epsilon)$ be weak solution  to problem $(P_{\omega^\epsilon}^f)$,  with $\epsilon \le 0$. Then,
$\epsilon\to u_\epsilon(\v_n)$ is   continuous and  decreasing in $(-\infty,0])$,    $\epsilon\to u_\epsilon(\v_1)$ is   continuous and  increasing in $(-\infty,0]$,  and
 \begin{equation}\label{1557dlii} 
 \begin{array}{l}\displaystyle
 \lim_{\epsilon\to-\infty}u_\epsilon(\v_n)=-\infty 
 \hbox{ or  } 
 \lim_{\epsilon\to+\infty}u_\epsilon(\v_1)=-\infty.
 \end{array}
 \end{equation}
 \end{itemize}
\end{lemma}

\begin{proof} Observe that $u_\epsilon(\v_n)= [u_\epsilon]_{\e_2}(\ell_{\e_2})$ and
$u_\epsilon(\v_1)=[u_\epsilon]_{\e_1}(0)$. The monotonicity of  $\epsilon\to u_\epsilon(\v_n)$ and $\epsilon\to u_\epsilon(\v_1)$ follows from Lemma~\ref{030326_002}. The main facts here are to prove  the continuity  and the convergence statements. 

 Let us call $[v_\epsilon]_{\e_i}=v_{\epsilon,i}$ and $[u_\epsilon]_{\e_i}=u_{\epsilon,i}$, $\ell_{\e_i}=\ell_i$, $\gamma_{\e_i}=\gamma_i$ and $p_{\e_i}=p_i$. As in the previous remark,  we use  Proposition~\ref{lemma2} to say that
$$v_{\epsilon,i}=\gamma_{i}(u_{\epsilon,i})  \ \hbox{is solution of problem} \ P^{\gamma_i, [f]_{\e_1}}_{ p_i, a_{\e_i}^\epsilon,b_{\e_i}^\epsilon} , \ i= 1,2,$$  for adequate $a_{\e_i}^\epsilon,b_{\e_i}^\epsilon$. Here,
$$  a_{\e_2}^\epsilon= -(|u_{\epsilon,2}'|^{p_2-2}u_{\epsilon,2}')(0), \quad  b_{\e_1}^\epsilon= (|u_{\epsilon,1}'|^{p_1-2}u_{\epsilon,1}')(\ell_{1}),$$
that we can rewrite as
$$a_{\e_2}^\epsilon=a_{\e_2}^0-\sigma_\epsilon,\ b_{\e_1}^\epsilon=b_{\e_1}^0+\mu_\epsilon,$$
where
$$\sigma_\epsilon=  (|u_{\epsilon,2}'|^{p_{2}-2}u_{\epsilon,2}')(0)
 -(|u_{0,2}'|^{p_{2}-2}u_{0,2}')(0),$$
 and
$$\mu_\epsilon=  (|u_{\epsilon,1}'|^{p_{1}-2}u_{\epsilon,1}')(\ell_{1})- (|u_{0,1}'|^{p_{1}-2}u_{0,1}')(\ell_{1}).$$
We also have that
$$\sigma_\epsilon= \omega_{\vi_{\e_2}}-(a_{\e_2}^\epsilon-a_{\e_2}^0)$$
and
$$\mu_\epsilon= \omega_{\vf_{\e_1}}-(b_{\e_1}^\epsilon-b_{\e_1}^0)$$

If $$\sigma_\epsilon>\epsilon \ \hbox{ or }\  \mu_\epsilon>\epsilon$$
we get a contradiction with  the continuity of the solution by using Lemma~\ref{lemma5} and  Lemma~\ref{030326_002}.  The same happens if $\sigma_\epsilon<0$ or $\mu_\epsilon<0$. So,
$$0\le \sigma_\epsilon,\, \mu_\epsilon\le \epsilon.$$
Under this conditions, continuity is proved as in the previous remark.

  Let us  prove~\eqref{1557dl} (the proof of~\eqref{1557dlii}  is similar). Here the discussion begins like in the previous remark.
Take $\epsilon_n$ a sequence such that $\lim_n\epsilon_n=+\infty$. Now, there exists a subsequence that we denote equal,  such that (i) or (ii) happen, where:

\noindent (i) the two following properties hold:
$$(|u_{\epsilon_n,2}'|^{p_{2}-2}u_{\epsilon_n,2}')(x)
 -(|u_{0,2}'|^{p_{2}-2}u_{0,2}')(x)>0\quad\hbox{ for all } 0<x<\ell_2,$$
and
$$(|u_{\epsilon_n,1}'|^{p_{1}-2}u_{\epsilon_n,1}')(x)
 -(|u_{0,1}'|^{p_{1}-2}u_{0,1}')(x)>0\quad\hbox{ for all } 0<x<\ell_1,$$

\noindent (ii) there exists $x_{\epsilon_n}\in (0,\ell_2)$, for each $n$, such that $$(|u_{\epsilon_n,2}'|^{p_{2}-2}u_{\epsilon_n,2}')(x)
 -(|u_{0,2}'|^{p_{2}-2}u_{0,2}')(x)>0\quad\hbox{for all }x_{\epsilon_n}<x<\ell_2,$$
 and 
 $$(|u_{\epsilon_n,2}'|^{p_{2}-2}u_{\epsilon_n,2}')(x_{\epsilon_n})
 -(|u_{0,2}'|^{p_{2}-2}u_{0,2}')(x_{\epsilon_n})=0;$$
 or there exists $x_{\epsilon_n}\in (0,\ell_1)$, for each $n$, such that $$(|u_{\epsilon_n,1}'|^{p_{1}-2}u_{\epsilon_n,1}')(x)
 -(|u_{0,1}'|^{p_{1}-2}u_{0,1}')(x)>0\quad\hbox{for all }x_{\epsilon_n}<x<\ell_1,$$
 and 
 $$(|u_{\epsilon_n,1}'|^{p_{1}-2}u_{\epsilon_n,1}')(x_{\epsilon_n})
 -(|u_{0,1}'|^{p_{1}-2}u_{0,1}')(x_{\epsilon_n})=0.$$

In the case (i), we have that $$u_{\epsilon_n,2}'(x)\ge u_{0,2}'(x) \quad\hbox{ for all } 0<x<\ell_2,$$
and integrating,
\begin{equation}\label{2m26001fl}
u_{\epsilon_n,2}(x)-u_{0,2}(x)\ge u_{\epsilon_n,2}(y)-u_{0,2}(y)\quad\forall\, 0\le y\le x\le\ell_2.
\end{equation}
And $$u_{\epsilon_n,1}'(x)\ge u_{0,1}'(x) \quad\hbox{ for all } 0<x<\ell_1,$$
and integrating,
\begin{equation}\label{2m26001p1fl}
u_{\epsilon_n,1}(x)- u_{0,1}(x)\ge u_{\epsilon_n,1}(y)-u_{0,2}(y)\quad\forall\, 0\le y\le x\le\ell_1.
\end{equation}

Now,
$$\int_0^{\ell_2}(v_{\epsilon_n,2}-v_{0,2})=\epsilon_n-\sigma_{\epsilon_n},$$
and  
$$\int_0^{\ell_1}(v_{\epsilon_n,1}-v_{0,1})=-\epsilon_n+\mu_{\epsilon_n},$$

Hence if $$\lim_n(\epsilon_n-\sigma_{\epsilon_n})=+\infty,$$ by the previous equality, we get that there exists $ y_{n}\in (0,\ell_2)$ such that
$$\lim_n v_{\epsilon_n,2}(y_n)=+\infty,$$
and also,
$$\lim_n u_{\epsilon_n,2}(y_n)=+\infty.$$
Then, by~\eqref{2m26001fl},
$$\lim_n u_{\epsilon_n,2}(\ell_2)=+\infty.$$
Similarly, 
$$\lim_n u_{\epsilon_n,1}(0)=-\infty$$
if 
$$\lim_n(-\epsilon_n+\mu_{\epsilon_n})=-\infty,$$

In the case 
$\lim_n(\epsilon_n-\sigma_{\epsilon_n})\not=+\infty $
and $\lim_n(-\epsilon_n+\mu_{\epsilon_n})\not=-\infty,$
we have that $$\lim_n\sigma_{\epsilon_n}=+\infty,$$
and
$$\lim_n\mu_{\epsilon_n}=+\infty.$$
Now,   on account of~\eqref{2m26001fl} and \eqref{2m26001p1fl},
$$u_{\epsilon_n,1}(y)-u_{0,1}(y)\le u_{\epsilon_n,1}(\ell_1)-u_{0,1}(\ell_1)=:\beta_n$$  and $$\alpha_n:=u_{\epsilon_n,2}(0)-u_{0,2}(0)\le u_{\epsilon_n,2}(x)-u_{0,2}(x)$$  for $0\le y\le\ell_1,\ 0\le x\le \ell_2$.
Now if $\sigma_\epsilon\ge \mu_\epsilon$,  by using Lemma~\ref{antepenultimo}, 
we have that $$\alpha_n \ge 0,$$
and if  $\sigma_\epsilon\ge \mu_\epsilon$
then
$$\beta_n \le 0.$$
Once we   arrive here, we finish as in the previous remark, also in the case  (ii). 
\end{proof}

The following theorem is the main statement of this section. It gives the existence of solutions to $(P_{\omega}^f)$.

\begin{theorem}\label{ExitUniq1}  Let  $\G$ be a   connected and compact metric graph.
  Let $1 < p_\e < \infty$ be and   $\gamma_\e:\R\to \R$ be a continuous and increasing function with $\gamma_\e(\mathbb{R})=\mathbb{R}$ and $\gamma_\e(0)=0$ for all $\e \in E(\G)$.  Then, for all   $f \in L^{1}(\G)$  and all $\omega=\{ \omega_\v: \v \in \V(\G) \} \subset \R$,   there exists a unique weak solution to the problem $(P_{\omega}^f)$.
\end{theorem}
 
\begin{proof} Uniqueness  is consequence of Theorem \ref{teocontprin001}. We will do the proof   of   existence  by induction on the number $N$ of edges. By Theorem~\ref{lemma1}, the result is true for $N=1$.  

Although it is not necessary we illustrate how we solve the case $N=2$, which gives an idea of the technical steps we do when passing from $N$ to $N+1$ in the induction procedure.   So, let $V(\G) = \{\v_1, \v_2, \v_3 \}$, $E(\G) = \{\e_1, \e_2 \}$, with $\v_1\in \e_1$, $\v_2\in \e_i$, $i=1,2$, and $\v_3\in \e_2$; that is, the graph is a linear graph with $\v_1,\v_3$ boundary vertices and $\v_2$ and interior vertex  (like in Figure~\ref{fig:002}). Let $b\in \mathbb{R}$, by Theorem~\ref{lemma1}, there exists a weak solution $ v^1_b=\gamma_{\e_1}(u^1_b)$ to problem ($P_{p_{\e_1},\omega_{\v_1},\omega_{\v_2}+b}^{\gamma_{{\e_1}},[f]_{\e_1}}$), and there exists  a weak solution $ v^2_b=\gamma_{\e_2}(u^2_b)$ to problem ($P_{p_{\e_2},-b,\omega_{\v_3}}^{\gamma_{{\e_2}},[f]_{\e_2}}$). In the case where $u^1_0(\ell_{\e_1}) = u^2_0(0)$,   by defining
$$  [v]_{\e_1}:= v^1_0,   \ [v]_{\e_2}:= v^2_0,$$
($[u]_{\e_1}:= u^1_0,   \ [u]_{\e_2}:= u^2_0$) we have that the function $v$ 
is the only weak solution to the problem~$(P_\omega^{f})$. Suppose now that $u^1_0(\ell_{\e_1}) \not= u^2_0(0)$, for example, $u^1_0(\ell_{\e_1}) < u^2_0(0)$. Then applying Lemma \ref{lemma3}, we can take $b>0$ large enough to reach $u^1_b(\ell_{\e_1}) = u^2_b(0)$, and consequently, {\it gluing the solutions,} we obtain   the only  weak solution to the problem  $(P_{\omega}^f)$, that is,   by defining  
$$   [v]_{\e_1}:= v^1_b,  \ [v]_{\e_2}:= v^2_b.$$
($[u]_{\e_1}:= u^1_b, \ [u]_{\e_2}:= u^2_b).$

By induction, suppose that the result is true for metric graphs with $N$ edges and let us see that it is true for metric graphs with $N+1$ edges.
So, let $\G_{N+1}$ be a metic graph with $\E(\G_{N+1}) = \{ \e_1, \ldots, \e_N, \e_{N+1} \}$, $\V(\G_ {N+1}) = \{ \v_1, \ldots, \v_m \}$ and $\omega = \{\omega_{\v_1}, \ldots, \omega_{\v_m} \}$. We want to solve $(P_\omega^f)$, with  $\overline{p}$ and $\overline{\gamma}$ also given. Let us now consider two situations:

\noindent (a) The graph $\G_{N+1}$ has at least a vertex with degree $1$  (that is, a boundary vertex). Suppose that such a vertex is $\v_m$ and that $\v_{m-1},\v_m\in \e_{N+1}$, and that the parametrization given in the formulation for  problem $(P_\omega^f)$ of $\e_{N+1}$ is such that $\vi_{\e_{N+1}}=\v_{m-1}$ and $\vf_{\e_{N+1}}=\v_m$.

Consider the graph $\G_N$ with $\E(\G_{N}) = \{ \e_1, \ldots, \e_N \}$, $\V(\G_{N}) = \{ \v_1, \ldots, \v_{m-1} \}$, and take $\epsilon\in \mathbb{R}$. For this graph we  consider the problem $(P_{\omega^\epsilon}^f)$ with $\omega^{\epsilon}$ given by 
$$\omega^\epsilon_{\v_i}=\omega_{\v_i},\ i=1,2,...,m-2,$$
$$\omega^\epsilon_{\v_{m-1}}=\omega_{\v_{m-1}}+\epsilon,$$
being the corresponding $f$, $\overline{p}$ and $\overline{\gamma}$ the same associated to $\e_i$ for problem $(P_\omega^f)$ given on the graph $\G_{N+1}$ (we avoid to introduce further notation for simplicity).

Consider also the graph given by one edge with initial vertex $\v_a$ and final vertex $\v_b$, parametrized in $(0,\ell_{\e_{N+1}})$. 
For this graph we consider the problem, written in terms of the parametrization,
$$(P_{p_{\e_{N+1}},-\epsilon,\omega_{\v_m}}^{\gamma_{\e_{N+1}},[f]_{\e_{N+1}}}).$$
 
By   hypothesis of induction, we have that there exists  a weak solution $\tilde v_{\epsilon}=\overline{\gamma}(\tilde u_{\epsilon})$ to the problem $(P_{\omega^\epsilon}^f)$. And we also have     a weak solution $ v_{\epsilon}=\gamma_{\e_{N+1}}(u_{\epsilon})$ to the problem  $(P_{p_{\e_{N+1}},-\epsilon,\omega_{\v_m}}^{\gamma_{\e_{N+1}},[f]_{\e_{N+1}}})$.

By Lemma \ref{lemma3}  and Lemma \ref{lemma5}, using the same idea given for the case $N=2$, there exists $\epsilon$ such that 
$$[\tilde{u}_{\epsilon}]_{\e_N}(\v_{m-1}) = u_{\epsilon}(\v_a).$$
 Then, gluing these solutions, vertex $\v_{m-1}$ is identified with $\v_a$, and $\v_m$ with $\v_b$, we get a solution to  problem $(P_\omega^f)$, we leave the details for the reader.

\noindent (b) All the vertices of the graph $\G_{N+1}$ have degree larger or equal than $2$  (that is, all the vertices are interior vertices).  
 Suppose that  $\v_{m-1},\v_m\in \e_{N+1}$, and that the parametrization given in the formulation for  problem $(P_\omega^f)$ of $\e_{N+1}$ is such that $\vi_{\e_{N+1}}=\v_{m-1}$ and $\vf_{\e_{N+1}}=\v_m$.

Consider the graph $\G_N$ with $\E(\G_{N}) = \{ \e_1, \ldots, \e_N \}$, $\V(\G_{N}) = \{ \v_1, \ldots, \v_{m} \}$ (the same vertices as for $G_{N+1}$, we are removing one edge).  This procedure may, in some situations, yield a disconnected graph, however, this does not pose a problem. In fact, this is a simpler case for the subsequent arguments (the details of which are left to the reader). Therefore, we  argue as if $\G_N$ were connected.

Take $\epsilon\in \mathbb{R}$. Following the idea  of step (a),  with two independent  gluing steps, we can solve problem $(P_{\omega^\epsilon}^f)$ for the graph $\widetilde{\G_N}$ with $\E(\widetilde{\G_N}) = \E(\G_{N})\cup \{\e_a,\e_b\}$, and $m+2$ vertices $\{\v_1, \ldots, \v_{m},\v_a,\v_b\}$, such that

-- $(\E(\G_N),\V(\G_{N}))$ is a subgraph of $\widetilde{\G_N}$,

-- $\e_a$ is an edge connecting $\v_{m-1}$ and $\v_a$, $\v_a$ is a boundary vertex, $\vi_{\e_a}=\v_{m-1}$, $\vf_{\e_a}=\v_a$, and $\e_a$ is parametrized in   $(0,\ell_{\e_{N+1}}/2)$, 

-- $\e_b$ is an edge connecting $\v_{m}$ and $\v_b$, $\v_b$ is a boundary vertex, $\vi_{\e_b}=\v_b$, $\vf_{\e_b}=\v_{m}$, and $\e_b$ is parametrized in $(0,\ell_{\e_{N+1}}/2)$,

-- $\omega^{\epsilon}$ given by 
$$\begin{array}{c}\displaystyle\omega^\epsilon_{\v_i}=\omega_{\v_i},\ i=1,2,...,m,\\ \\
\displaystyle\omega^\epsilon_{\v_a}=\epsilon, \quad 
\omega^\epsilon_{\v_b}=-\epsilon,
\end{array}$$
and being the corresponding $f$, $\overline{p}$ and $\overline{\gamma}$ the same associated to $\e_i$, $i=1,2,...,N$, for problem $(P_\omega^f)$ given on the graph $\G_{N+1}$, and 

-- for $\e_a$: $p_{\e_a}=p_{N+1}$, $\gamma_{\e_a}=\gamma_{N+1}$, $[f]_{\e_a}= [f]_{\e_{N+1}}$; and 

-- for $\e_b$: $p_{\e_b}=p_{N+1}$, $\gamma_{\e_b}=\gamma_{N+1}$, $[f]_{\e_b}(x)= [f]_{\e_{N+1}}(x+\ell_{\e_{N+1}}/2)$ for $x\in (0,\ell_{N+1}/2)$.

Then, applying Lemma~\ref{lemma51822}, we can find $\epsilon$ such that  the  solution $v^\epsilon=\gamma_{N+1}(u^\epsilon)$ of the above problem satisfies $$[u^\epsilon]_{\e_a}(\ell_{\e_{N+1}}/2)=[u^\epsilon]_{\e_b}(0).$$
Hence, by joining $\v_a$ and $\v_b$, we can glue the above solutions to obtain a solution to problem $(P_\omega^f)$.
\end{proof}

\begin{remark}\rm  Observe that in the above proof we  construct solutions for a complex quantum  graph by gluing solutions of {\it smaller}    graphs. But, once we have shown existence and uniqueness of solutions  for problem~$(P_\omega^f)$, on account of Proposition~\ref{lemma2}, we can construct the solution in the following natural way.  For each edge $\e\in\E(\G)$, solve the $1$-dimensional ode problem ($P_{p_\e,a_\e,b_\e}^{\gamma_\e,[f]_\e}$) for arbitrary $a_\e,b_\e$, and get a solution $\gamma_\e(u_\e^{a_\e,b_\e})$ (it is unique for each pair $a_\e,b_\e$). Then, solve the system  that gives the Neumann-Kirchhoff condition, that is,
\begin{equation}\label{wwg1333otro}
\left\{\begin{array}{l}
\displaystyle \sum_{\e:\vi_\e=\v}a_\e+\sum_{\e:\vf_\e=\v}b_\e=\omega_\v,\quad \v\in \V(\G),\\ \\
\displaystyle u_\e^{a_\e,b_\e}(0)=u_{\tilde \e}^{a_{\tilde \e},b_{\tilde \e}}(0),\quad \hbox{for }  \e,{\tilde \e}:\vi_\e=\vi_{\tilde \e}\,,
\\ \\
\displaystyle u_\e^{a_\e,b_\e}(0)=u_{\tilde \e}^{a_{\tilde \e},b_{\tilde \e}}(\ell_{\tilde\e}),\quad \hbox{for } \e,{\tilde \e}:\vi_\e=\vf_{\tilde \e}\,,
\\ \\
\displaystyle u_\e^{a_\e,b_\e}(\ell_\e)=u_{\tilde \e}^{a_{\tilde \e},b_{\tilde \e}}(\ell_{\tilde\e})\quad \hbox{for } \e,{\tilde \e}:\vf_\e=\vf_{\tilde \e}\,,
\end{array}\right.
\end{equation}
to find $\{a_\e,b_\e\}_{\e\in\E(\G)}$.  Finally, the solution of~$(P_\omega^f)$ is $\overline{\gamma}(u)$ where $u$ is given by $[u]_\e:=u_\e^{a_\e,b_\e}$, for all $\e\in\E(\G)$. 
 The existence and uniqueness of solutions for~$(P_\omega^f)$ guarantees that the above system is solvable and has a unique solution.  \hfill $\blacksquare$
\end{remark}

\begin{proposition}\label{laacotLinfdeleliptico001}
   Let  $\G$ be a  connected and compact metric graph.
  Let $1 < p_\e < \infty$ be and   $\gamma_\e:\R\to \R$ be a continuous and increasing function with $\gamma_\e(\mathbb{R})=\mathbb{R}$ and $\gamma_\e(0)=0$ for all $\e \in E(\G)$.
   For $v = \overline{\gamma}(u)$  the weak solution  to $(P_{0}^{f})$ (in this case $\omega_\v=0$ for all $\v$), we have  
$$ v\ll v +\lambda(f-v)\ \hbox{  for all }\lambda>0.$$
In particular, for $f\in L^\infty(\G)$, $||v||_\infty\le ||f||_\infty$.
\end{proposition}

\begin{proof}
Let us prove that
$$\int_\G(f-v)\chi_{\{v<-h\}}dx\le 0\le  \int_\G (f-v)\chi_{\{v>h\}}dx\quad\forall h>0.$$ This implies that (see~Proposition~\ref{inq1BC})
$ v\ll v +\lambda(f-v)$ for all $\lambda>0$; in particular, for $\lambda=1$, we get $v\ll f$, and $||v||_\infty\le ||f||_\infty$ (see again~Proposition~\ref{inq1BC}).

Take $h>0$ and set $k=\gamma^{-1}(h)$. For $\epsilon>0$, taking $\varphi=\frac{1}{\epsilon}T_\epsilon(u-k)^+$ as test function in \eqref{Weak1}, and  mislead the non-negative  terms involved with $|[u]_\e'|^{p_\e-2}[u]_\e'$, we  get
$$\int_G(f-v)\frac{1}{\epsilon}T_\epsilon(u-k)^+dx\ge 0.$$
Now, taking limits as $\epsilon\to 0$ we obtain that 
$$\int_\G (f-v)\chi_{\{u>k\}} dx\ge 0$$
or equivalently,
$$\int_\G (f-v)\chi_{\{v>h\}} dx\ge 0.$$
The another inequality follows in a similar way.
\end{proof}

\subsection{The parabolic problem on a metric graph.}
    The aim of this section is to solve first the parabolic Problem~\eqref{evpbintro1711} in the nonlinear semigroup sense and, afterwards, characterize the solutions obtained in this way as the unique weak solutions of the problem. 

 Observe that, on account of the notation introduced in   Remark~\ref{defdelplaplacianoenG} we can rewrite Problem~\eqref{evpbintro1711} as
 $$  \hbox{\eqref{evpbintro1711}}\quad \left\{\begin{array}{l}\displaystyle
\frac{\partial v}{\partial t}-\Delta^{\G,\omega}_{\overline{p}}u=f\quad\hbox{in } (0,+\infty)\times \G,\\
\\{}
v=\overline{\gamma}(u)\quad\hbox{in } (0,+\infty)\times\G,
\\ \\
\hbox{$u$ continuous,}
\\ \\
v(0)=v_0.
\end{array}\right.$$

Set $X:=L^{1}(\G)\times L^1(\V(\G))$. Here $L^1(\V(\G))$ is indeed the finite space of functions from $\V(\G)$ to~$\mathbb{R}$, that is $\mathbb{R}^m$ with $m=\hbox{card}(\V( \G))$, that we consider with the $1$-norm. For a function $\omega\in L^1(\V(\G))$ we will write $\omega(\v)$ also as $\omega_\v$.

Let    $f\in L^1(0,T;L^{1}(\G))$ and $\omega\in L^1(0,T;L^1(\V(\G)))$. Problem~\eqref{evpbintro1711},  posed in $(0,T)$, can be rewritten 
as  the following abstract Cauchy problem in $X$:
\begin{equation}\label{ACPI}
\left\{ \begin{array}{ll} W'(t) + \mathcal{A} (W(t)) \ni (f(t),\omega(t)) \quad &  t \in (0,T), \\[10pt] W(0) = (v_0,0) \quad & v_0  \in L^1(\G),   \end{array} \right.
\end{equation}
being $\mathcal{A} $ the operator:
$$
\mathcal{A}  =  \{((v,0), (\widetilde v,\omega)) \in  X\times X  :   v \hbox{ is a solution of } (P_{\omega}^{v+\widetilde v})\}.$$
   We have that $$((v,0), (\widetilde v,\omega)) \in \mathcal{A} \iff \hbox{there exist } u\in \mathcal{W}^{1,\overline{p}}(\G), \ v=\gamma(u)\ \hbox{ a.e. in } \G,$$
such that
 $$-\Delta_{\overline{p}}^{\G,\omega}u=\tilde v\,,$$ that is
\begin{equation}\label{Weak1rep01} \sum_{\e \in\E(\G)}\int_{0}^{\ell_\e}|[u]_\e'|^{p_\e-2}[u]_\e' [\varphi]_\e' \, \,dx  = \int_\G \widetilde v \varphi \, \,dx +   \sum_{\v\in \V(\G)}\omega_\v\varphi(\v)\quad \forall \, \varphi \in \mathcal{W}^{1,\overline{p}}(\G). \end{equation}

\begin{remark}\rm 
For standard Neumann-Kirchhoff conditions, that is, when $$\omega_\v=0\quad\forall\v\in\V(\G),$$ Problem~\eqref{evpbintro1711} can be rewritten  as the abstract Cauchy problem in $L^{1}(\G)$:
\begin{equation}\label{ACPIsub0}
\left\{ \begin{array}{ll} V'(t) + \mathcal{A}_0 (V(t)) \ni f \quad &  t \in (0,T), \\[10pt] V(0) = v_0 \quad & v_0 \in L^1(\G),   \end{array} \right.
\end{equation}
being now $\mathcal{A}_0 $ the operator:
$$
\mathcal{A}_0  =  \{(v, \widetilde v) \in L^{1}(\G) \times L^{1}(\G)   :  v \hbox{ a solution of } (P_{0}^{v+\widetilde v})\}.$$
Here we have that $(v,\tilde v)\in \mathcal{A}_0$ iff there exists $u\in \mathcal{W}^{1,\overline{p}}(\G), \ v=\gamma(u)  \hbox{ a.e. in } \G,$
such that
 $-\Delta_{\overline{p}}^{\G,0}u=\tilde v$. In this case, some of the arguments that follow are simpler than in the general case.

It is worthy to mention than dealing with non-null Neumann-Kirchhoff conditions has an extra difficulty that, following~\cite{AndIgMazTo06}, can be solved by considering an operator in  $X\times X$ in the way we have done. We will adapt some arguments of our paper~\cite{AndIgMazTo06}  to this scenario.  
   \hfill $\blacksquare$
\end{remark}
 
\begin{theorem}\label{Dominio} Under the assumptions of Theorem~\ref{ExitUniq1} we have that the operator $\mathcal{A}$ is  $m$-T-accretive and satisfies 
\begin{equation}\label{qdomaindensity01}
\overline{D(\mathcal{A})}=L^{1}(\G)\times\{0\}.
\end{equation}
\end{theorem}
\begin{proof}  The $m$-T-accretivity of the operator $\mathcal{A}$ is consequence of Theorem~\ref{ExitUniq1}  and Theorem~\ref{teocontprin001}. 

 Let us prove~\eqref{qdomaindensity01}.   Given $g \in L^\infty(\G)$,  by Theorem~\ref{ExitUniq1},  we have that there exist $(v_n,0) \in D(\mathcal{A})$ such that $$(v_n,0) = \left(I + \frac{1}{n} \mathcal{A}\right)^{-1}(g,0).$$ 
Let us see that $(v_n,0)\to (g,0)$ in $X$, that is, $v_n\to g$ in $L^1(\G)$.

We have that
$((v_n,0), (n(g - v_n),0)) \in \mathcal{A}$. 
Hence,
\begin{equation}\label{masclaro}
v_n \ \hbox{is a weak solution of problem} \ (P_{0}^{v_n+n(g-v_n)}). 
\end{equation}
Thus, there exists $u_n \in \mathcal{W}^{1,\overline{p}}(\G)$, with $v_n = \overline{\gamma}(u_n)$, satisfying
\begin{equation}\label{Weak1rep01NN} \sum_{\e \in\E(\G)}\int_{0}^{\ell_\e}|[u_n]_\e'|^{p_\e-2}[u_n]_\e' [\varphi]_\e' \, \,dx  = n \int_\G (g - v_n) \varphi \, \,dx  \quad \forall \, \varphi \in \mathcal{W}^{1,\overline{p}}(\G). \end{equation}
Hence,   for   $\varphi \in \mathcal{W}^{1,\overline{p}}(\G)$, by H\"older's inequality, we have:
$$\left\vert \int_\G (g - v_n) \varphi \, \,dx \right\vert \leq \frac{1}{n} \sum_{\e \in\E(\G)}\int_{0}^{\ell_\e}|[u_n]_\e'|^{p_\e-1} \vert [\varphi]_\e'\vert \, \,dx $$
$$\leq  \sum_{\e \in\E(\G)}  \left( \frac{1}{n}\int_{0}^{\ell_\e}|[u_n]_\e'|^{p_\e} \right)^{\frac{p_\e -1}{p_\e}} \left( \frac{1}{n}\int_{0}^{\ell_\e} \vert [\varphi]_\e'\vert^{p_\e} \right)^{\frac{1}{p_\e}}$$
$$\le  \frac{M_1}{n^{1/p_{\infty}}}\sum_{\e \in\E(\G)}  \left( \frac{1}{n}\int_{0}^{\ell_\e}|[u_n]_\e'|^{p_\e} \right)^{\frac{p_\e -1}{p_\e}}$$
where $p_{\infty}=\max\{p_\e:\e\in \E(\G)\}$, and $M_1$ depends on $\varphi$.
Now, taking $\varphi = u_n$ in \eqref{Weak1rep01NN}, and taking into account Proposition~\ref{laacotLinfdeleliptico001}  and that $u_nv_n\ge 0$, we get
\begin{equation}\label{theobveest001}\frac{1}{n}\sum_{\e \in\E(\G)}\int_{0}^{\ell_\e}|[u_n]_\e'|^{p_\e}   = \int_\G (g - v_n) u_n \,dx    \leq  \int_\G g u_n   \,dx .\end{equation}

From \eqref{masclaro} and~Proposition~\ref{laacotLinfdeleliptico001}, we have    $v_n\ll v_n+\lambda n(g-v_n)$ for all $\lambda>0$; in particular, for $\lambda=\frac{1}{n}$, we get 
\begin{equation}\label{theobveest001otra02}v_n\ll g,
\end{equation}
and therefore $\Vert v_n\Vert_\infty\le \Vert g\Vert_\infty$ (see Proposition~\ref{inq1BC}). Hence, from~\eqref{theobveest001},
\begin{equation}\label{theobveest001otra01}\frac{1}{n}\sum_{\e \in\E(\G)}\int_{0}^{\ell_\e}|[u_n]_\e'|^{p_\e}       \leq M_2.\end{equation}

Then, 
$$\begin{array}{c}
\displaystyle\sum_{\e \in\E(\G)}  \left( \frac{1}{n} \int_{0}^{\ell_\e}|[u_n]_\e'|^{p_\e} \right)^{\frac{p_\e -1}{p_\e}} \leq \sharp\left\{\e\in E(\G) :  \int_{0}^{\ell_\e}|[u_n]_\e'|^{p_\e} \leq 1 \right\} + \frac{1}{n}\sum_{\e \in\E(\G)}\int_{0}^{\ell_\e}|[u_n]_\e'|^{p_\e}\\ \\
\leq \sharp(E(\G)) +M_2.
\end{array}$$
And therefore, 
$$\left\vert \int_\G (g - v_n) \varphi \, \,dx \right\vert \leq   \frac{M_3}{n^{1/p_{\infty}}},$$
where $M_3$ depends on $\varphi$.
Consequently,
\begin{equation}\label{limweak1}
\lim_{n \to \infty}  \int_\G (g - v_n) \varphi \, \,dx =0 \quad \forall \, \varphi   \in \mathcal{W}^{1,\overline{p}}(\G).
\end{equation}

 From \eqref{theobveest001otra02} and Proposition~\ref{inq1BC}, we have that there exists a subsequence $\{v_{n_k}\}_k$   weakly convergent in $L^1(\G)$, so, by~\eqref{limweak1}, it converges to $g$. Therefore, since $v_{n_k}\ll g$, applying again  Proposition~\ref{inq1BC}, we have that 
 $$v_{n_k}\to g\quad\hbox{in } L^{1}(\G).$$
Consequently, we have shown that $L^\infty(\G)\times\{0\}\subset \overline{D(\mathcal{A})},$ and therefore \eqref{qdomaindensity01} holds.\end{proof}

  We will use the following notation:  for $1\le q\le \infty$, $$L^q_{loc}([0,+\infty);Y)=\{h:(0,+\infty)\to Y:h\in L^q(0,T;Y)  \ \forall T>0\};$$  
  we will also use these other notations:
  $$L^{\overline{p'}}_{loc}([0,+\infty);  L^{\overline{p'}}(\G))
=\{h:(0,+\infty)\to G: [h]_\e\in L^{p_\e'}(0,T;L^{p_\e'}(0,\ell_\e))\ \forall T>0 \   \&\ \forall \e\},$$
and, for $T>0$,
$$L^{\overline{p}}((0,T)\times\G)=\{h\in L^1(0,T;L^{\overline{p}}(\G)): [h]_\e\in L^{p_\e}(0,T;L^{p_\e}(0,\ell_\e)) \ \forall \e\}$$
  $$L^{\overline{p}}(0,T;  \mathcal{W}^{1,\overline{p}}(\G))=\{h\in L^{1}(0,T;  \mathcal{W}^{1,\overline{p}}(\G)): [h]_\e\in L^{p_\e}(0,T;W^{1,p_\e}(0,\ell_\e))\ \forall \e\}.$$

By Theorem \ref{Dominio}, applying Theorem \ref{EUm-accretive}, we have the following result. 

 \begin{theorem}\label{Existencemild}  Let  $\G$ be a   connected and compact metric graph.
  Let $1 < p_\e < \infty$ be and   $\gamma_\e:\R\to \R$ be a continuous and increasing function with $\gamma_\e(\mathbb{R})=\mathbb{R}$ and $\gamma_\e(0)=0$ for all $\e \in E(\G)$.  Let $f\in L^1_{loc}([0,+\infty);L^{1}(\G))$  and $\omega\in L^1_{loc}([0,+\infty);L^1(\V(\G)))$ be. Then, for every initial data $v_0 \in L^1(\G)$, 
there exists a unique mild solution of Problem~\eqref{ACPI}  on $[0,T]$,    for any $T>0$.
 \end{theorem}

 Our next aim is to characterize this mild solution as a weak solution.
 
\begin{definition}  Let $f\in L^1_{loc}([0,+\infty);L^{1}(\G))$  and $\omega\in L^1_{loc}([0,+\infty);L^1(\V(\G)))$ be. We say that $v:[0,+\infty)\to L^1(\G)$ is a weak solution  of Problem~\eqref{evpbintro1711} if, for any $T>0$, $v \in C([0,T]; L^1(\G))$,  $v(0) = v_0$ and  there exists  $u \in L^{\overline{p}}(0,T;  \mathcal{W}^{1,\overline{p}}(\G))$ with $v(t) = \overline{\gamma}(u(t))$ for all $t \in [0,T]$ satisfying
\begin{equation}\label{Def1} \begin{array}{c} - \displaystyle\int_0^T \int_\G v(t) \frac{\partial}{\partial t} \psi(t)\,dx\, dt +  \sum_{\e \in\E(\G)}\int_0^T  \int_{0}^{\ell_\e}|[u(t)]_\e'|^{p_\e-2}[u(t)]_\e' [\psi(t)]_\e'  \,dx\, dt \\ \\
 \displaystyle = \int_0^T \int_{\G} f(t) \psi(t)\,dx\, dt+\sum_{\v\in \V(\G)}\int_0^T\omega_\v(t)\psi(t,\v) \, dt, \end{array}
\end{equation}
for all $\psi \in W^{1,1}(0,T; \mathcal{W}^{1,{1}}(\G)) \cap L^{\overline{p}}(0,T; \mathcal{W}^{1,\overline{p}}(\G))$ with $\psi(0) = \psi (T) =0$.
\end{definition}
\begin{remark}\rm 
A weak solution also satisfies
\begin{equation}\label{Def2}\begin{array}{c}\displaystyle  \frac{d}{dt}  \int_\G v(t) \varphi\,dx  +    \sum_{\e \in\E(\G)}\int_{0}^{\ell_\e}|[u(t)]_\e'|^{p_\e-2}[u(t)]_\e' [\varphi]_\e'\,dx  \\ \\   \displaystyle = \int_{\G} f(t) \varphi\,dx+\sum_{\v\in \V(\G)}\omega_\v(t)\varphi(\v) \quad \hbox{in} \ \ \mathcal{D}^\prime(]0,T[),
\end{array}
\end{equation}
for all $\varphi \in  \mathcal{W}^{1,\overline{p}}(\G)$.
In fact, we can replace~\eqref{Def1} by~\eqref{Def2} in the definition of weak solution by the density of the functions of the form  $\psi := \varphi \xi$, with  $\varphi \in \mathcal{W}^{1,\overline{p}}(\G)$ and $\xi \in \mathcal{D}(]0, T[)$, in $ L^{\overline{p}}(0,T; \mathcal{W}^{1,\overline{p}}(\G))$.  \hfill $\blacksquare$
\end{remark}

\begin{theorem} \label{uniqweak001}
Let  $\G$ be a  connected and compact metric graph.
  Let $1 < p_\e < \infty$ be and   $\gamma_\e:\R\to \R$ be a continuous and increasing function with $\gamma_\e(\mathbb{R})=\mathbb{R}$ and $\gamma_\e(0)=0$ for all $\e \in E(\G)$.  Let $v_0 \in L^1(\G)$, $f\in L^1_{loc}([0,+\infty);L^{1}(\G))$  and $\omega\in L^1_{loc}([0,+\infty);L^1(\V(\G)))$ be. 
If $v$ is a  weak solution of Problem~\eqref{evpbintro1711}, then $W:= (v,0)$ is an integral solution of Problem~\eqref{ACPI}. 
\end{theorem}

\begin{proof}   Firstly observe that the bracket in  $X =L^{1}(\G)\times L^1(\V(\G))$ is given by 
$$\begin{array}{c}\displaystyle[(u_1,w_1),(u_2,w_2)] = \int_\G u_2 \, \hbox{sign}_0 ( u_1) \,  dx + \int_{\{x\in \G:u_1(x) =0 \}} \vert u_2(x) \vert \, dx \\ \\
\displaystyle+ \sum_{\v \in V(\G)} w_2(\v) \,  \hbox{sign}_0 ( w_1(\v))  + \sum_{\v:w_1(\v) =0}  \vert w_2(\v) \vert.
\end{array}$$ 

 Let $v$ be a weak solution of Problem~\eqref{evpbintro1711}.
To prove that $W:= (v,0)$ is an integral solution of Problem~\eqref{ACPI},  we must  show that, for any $(\tilde{v},\tilde{w}) \in \mathcal{A}(\tilde{z}, 0)$, 
\begin{equation}\label{E2IntWeakneeds01}   \begin{array}{c} 
\displaystyle \frac{d}{dt} \int_\G \vert v(t) - \tilde{z} \vert \leq \int_\G (f(t)- \tilde{v}) \, \hbox{sign}_0 ( v(t)- \tilde{z}) \,  dx     \qquad \\ \\
\displaystyle \qquad + \int_{\{x\in\G:v(t,x) = \tilde{z}(x) \}} \vert  f(t,x)- \tilde{v} \vert \, dx
+ \sum_{\v \in V(\G)) }  \vert \omega_\v(t)  - \tilde{w}_\v\vert
\end{array}
\end{equation}
in $\mathcal{D}^{\prime}(]0,T[)$.

Since  $v$ is a weak solution of Problem~\eqref{evpbintro1711}, there exists $u \in L^{\overline{p}}(0,T;  \mathcal{W}^{1,\overline{p}}(\G))$ with $v(t) = \overline{\gamma}(u(t))$ for all $t \in [0,T]$ satisfying 
\begin{equation}\label{Def1proof01} \begin{array}{c} - \displaystyle\int_0^T \int_\G v(t) \frac{\partial}{\partial t} \psi(t)\,dx\, dt +  \sum_{\e \in\E(\G)}\int_0^T  \int_{0}^{\ell_\e}|[u(t)]_\e'|^{p_\e-2}[u(t)]_\e' [\psi(t)]_\e'  \,dx\, dt \\ \\
 \displaystyle = \int_0^T \int_{\G} f(t) \psi(t)\,dx\, dt+\sum_{\v\in \V(\G)}\int_0^T\omega_\v(t)\psi(t,\v) \, dt, \end{array}
\end{equation}
for all $\psi \in W^{1,1}(0,T; \mathcal{W}^{1,1}(\G)) \cap L^{\overline{p}}(0,T; \mathcal{W}^{1,\overline{p}}(\G))$ with $\psi(0) = \psi (T) =0$.
On the other hand, if $(\tilde{v},\tilde{w}) \in \mathcal{A}(\tilde{z}, 0),$ there exist $\tilde{u} \in \mathcal{W}^{1,\overline{p}}(\G)$, $ \tilde z = \overline{\gamma}(\tilde{u})$ in  $\G,$
such that
\begin{equation}\label{Weak1rep01N} \sum_{\e \in\E(\G)}\int_{0}^{\ell_\e}|[\tilde{u}]_\e'|^{p_\e-2}[\tilde{u}]_\e' [\varphi]_\e' \, \,dx  = \int_\G  \widetilde v \varphi \, \,dx +   \sum_{\v\in \V(\G)} \tilde{w}_\v\varphi_\v\quad \forall \, \varphi \in \mathcal{W}^{1,\overline{p}}(\G). \end{equation}
 Therefore,
\begin{equation}\label{Def1proof02} \begin{array}{c}   \displaystyle - \underbrace{\int_0^T \int_\G v(t) \frac{\partial}{\partial t} \psi(t)\,dx\, dt}_{(A)} \\ \\
\displaystyle
+  \underbrace{\sum_{\e \in\E(\G)}\int_0^T  \int_{0}^{\ell_\e}\Big(|[u(t)]_\e'|^{p_\e-2}[u(t)]_\e'-|[\tilde{u}]_\e'|^{p_\e-2}[\tilde{u}]_\e'   \Big)[\psi(t)]_\e'  \,dx\, dt}_{(B)} \\ \\
 \displaystyle = \underbrace{\int_0^T \int_{\G} \big(f(t)-\tilde v\big) \psi(t)\,dx\, dt}_{(C)}+\underbrace{\sum_{\v\in \V(\G)}\int_0^T\big(\omega_\v(t)-\tilde \omega_\v)\psi(t,\v) \, dt}_{(D)}, \end{array}
\end{equation}
for all $\psi \in W^{1,1}(0,T; \mathcal{W}^{1,1}(\G)) \cap L^{\overline{p}}(0,T; \mathcal{W}^{1,\overline{p}}(\G))$ with $\psi(0) = \psi (T) =0$.
Then for $\xi\in  \mathcal{D}^{\prime}(]0,T[)$, $\xi\ge 0$, we can take $\displaystyle\psi_\tau(x,t)=\frac{1}{\tau}\int_t^{t+\tau}\frac{1}{k}T_k(u(s,x)-\tilde u(x))\xi(s)\, ds$,
with  $\tau>0 $  and small, as test function in~\eqref{Def1proof02}.
Let us now study term by term the  equation we obtain with such test function.

  Respect to (A), we have
$$\int_\G v(t) \frac{\partial}{\partial t} \psi_\tau(t)\,dx\, dt = \int_0^T \int_\G v(t) \frac{1}{k}  \frac{T_k(u(t+\tau)-\tilde u)\xi(t+\tau)-T_k(u(t)-\tilde u)\xi(t)}{\tau}\,dx\, dt$$ $$=\int_0^T \int_\G \frac{v(t-\tau)-v(t)}{\tau}  \frac{1}{k} T_k(u(t)-\tilde u) \,dx\,\xi(t)\, dt.$$
Now, since for every $\e \in E(\G)$, the function 
$$r \mapsto   \int_0^r T_k((\gamma_\e)^{-1}(s)- [\tilde u]_\e)) ds $$
is non-decreasing, we have
$$  T_k((\gamma_\e)^{-1}(r)- [\tilde u]_\e)) \in \partial \left(   \int_0^r T_k((\gamma_\e)^{-1}(s)- [\tilde u]_\e)) ds \right).$$
Moreover,
$$ T_k([u(t)]_\e(x)-[\tilde u]_\e(x)) =    T_k((\gamma_\e)^{-1}([v(t)]_\e(x)])-[\tilde u]_\e(x)).$$
Therefore,
$$\begin{array}{c}
\displaystyle([v(t - \tau) ]_\e(x) - [v(t)]_\e(x) ) \frac{1}{k} T_k([u(t)]_\e(x)-[\tilde u]_\e(x)) \\
\\ \displaystyle\leq \int_{[v(t)]_\e(x) }^{([v (t - \tau)]_\e(x)} \frac{1}{k}  T_k((\gamma_\e)^{-1}(s)- [\tilde u]_\e(x)) ds.
\end{array}$$
Hence,
$$\begin{array}{c}
\displaystyle\int_0^T \int_\G \frac{v(t-\tau)-v(t)}{\tau}  \frac{1}{k} T_k(u(t)-\tilde u) \,dx\, \xi(t)\, dt 
\\ \\ \displaystyle
\leq\sum_{e \in E(\G)} \int_0^T  \int_0^{\ell_\e}\frac{1}{\tau}\int_{[v(t)]_\e(x) }^{([v(t - \tau)]_\e(x) }  \frac{1}{k}  T_k((\gamma_\e)^{-1}(s)- [\tilde u]_\e(x))) ds  \, dx \,  \xi(t)\, dt
\\ \\ \displaystyle
= \sum_{e \in E(\G)}\int_0^T \int_0^{\ell_\e} \int_{[\tilde z]_\e(x)}^{[v(t)]_\e (x)}\frac{1}{k}  T_k((\gamma_\e)^{-1}(s)- [\tilde u]_\e(x))) ds \, dx \,  \frac{\xi(t+\tau) - \xi(t)}{\tau}\,dt.
\end{array}
$$
So, 
$$
\begin{array}{c}
\displaystyle
\limsup_{\tau\to 0} \int_\G v(t) \frac{\partial}{\partial t} \psi_\tau(t)\,dx\, dt 
\\ \\ \displaystyle
\leq  \sum_{e \in E(\G)}\int_0^T\int_0^{\ell_\e} \int_{[\tilde z]_\e(x)}^{[v(t)]_\e (x)}\frac{1}{k}  T_k((\gamma_\e)^{-1}(s)- [\tilde u]_\e(x))) ds \, dx \,   \frac{d}{dt}\xi(t)\,dt.
\end{array}
$$
Now 
$$
\begin{array}{c}\displaystyle
\limsup_{k\to 0}  \int_{[\tilde z]_\e(x)}^{[v(t)]_\e (x)}\frac{1}{k}  T_k((\gamma_\e)^{-1}(s)- [\tilde u]_\e(x))) ds \le \int_{[\tilde z]_\e(x)}^{[v(t)]_\e (x)}  \hbox{sign}_0 ((\gamma_\e)^{-1}(s)- [\tilde u]_\e(x)) ds
\\ \\ \displaystyle
=  \int_{[\tilde z]_\e(x)}^{[v(t)]_\e (x)}  \hbox{sign}_0 ((\gamma_\e)^{-1}(s)- (\gamma_\e)^{-1}([\tilde z]_\e(x))) ds = \vert [v(t)]_\e (x) -  [\tilde z]_\e(x) \vert.
\end{array}
$$
Consequently
\begin{equation}\label{e1A}
 \displaystyle 
   \displaystyle \limsup_{k\to 0}\limsup_{\tau\to 0}\int_0^T \int_\G v(t) \frac{\partial}{\partial t} \psi_\tau(t)\,dx\, dt \leq  \int_0^T \int_\G  \vert v(t) - \tilde{z} \vert \,  \,dx\, \frac{d}{dt}\xi(t)\,dt.
\end{equation}
For (B), we have
$$ \begin{array}{ll}
  \displaystyle\limsup_{\tau\to 0}  \sum_{\e \in\E(\G)}\int_0^T  \int_{0}^{\ell_\e}\Big(|[u(t)]_\e'|^{p_\e-2}[u(t)]_\e'-|[\tilde{u}]_\e'|^{p_\e-2}[\tilde{u}]_\e' [\varphi]_\e' \Big)[\psi_\tau(t)]_\e'  \,dx\, dt\\ \\ 
    \displaystyle=\sum_{\e \in\E(\G)}\int_0^T  \int_{0}^{\ell_\e}\Big(|[u(t)]_\e'|^{p_\e-2}[u(t)]_\e'-|[\tilde{u}]_\e'|^{p_\e-2}[\tilde{u}]_\e'   \Big)\Big(\frac{1}{k}T_k([u(t)]_\e-[\tilde{u}]_\e)\Big)'\,dx\,\xi(t)\,dt\ge 0.
  \end{array}
$$
Hence
\begin{equation}\label{e1B}
\limsup_{k\to 0}\limsup_{\tau\to 0}\sum_{\e \in\E(\G)}\int_0^T  \int_{0}^{\ell_\e}\Big(|[u(t)]_\e'|^{p_\e-2}[u(t)]_\e'-|[\tilde{u}]_\e'|^{p_\e-2}[\tilde{u}]_\e' [\varphi]_\e' \Big)[\psi_\tau(t)]_\e'  \,dx\, dt \geq 0.
\end{equation}
In the case of (C), we have
\begin{equation}\label{e1C}\begin{array}{c}  
  \displaystyle\limsup_{k\to 0}\limsup_{\tau\to 0}\int_0^T \int_{\G} \big(f(t)-\tilde v\big) \psi_\tau(t)\,dx\, dt\\ 
  \\  \displaystyle=
  \limsup_{k\to 0}\int_0^T \int_{\G} \big(f(t)-\tilde v\big) \frac{1}{k}T_k(u(t,x)-\tilde u(x))\,dx\, \xi(t)\,dt\\ \\
     \displaystyle \le \int_0^T \int_\G (f(t)- \tilde{v}) \, \hbox{sign}_0 ( u(t)- \tilde u) \,  dx \,  \xi(t)\,dt \qquad\qquad \\ \\ 
      \displaystyle \qquad\qquad + \int_0^T\int_{\{x\in\G:u(t,x) = \tilde u(x) \}} \vert  f(t,x)- \tilde{v} \vert \, dx\, \xi(t)\,dt
   \\ \\
     \displaystyle = \int_0^T \int_\G (f(t)- \tilde{v}) \, \hbox{sign}_0 ( v(t)- \tilde z) \,  dx \,  \xi(t)\,dt  \qquad\qquad \\ \\ 
      \displaystyle \qquad\qquad + \int_0^T\int_{\{x\in\G:v(t,x) = \tilde z(x) \}} \vert  f(t,x)- \tilde{v} \vert \, dx\, \xi(t)\,dt.
  \end{array}
 \end{equation}
 And finally, respect to (D)
  \begin{equation}\label{e1D}\begin{array}{c} 
  \displaystyle    \limsup_{\tau\to 0} \sum_{\v\in \V(\G)}\int_0^T\big(\omega_\v(t)-\tilde \omega_\v\big)\psi_\tau(t,\v) \, dt, 
  \\ \\ \displaystyle=  
   \sum_{\v\in \V(\G)}\int_0^T\big(\omega_\v(t)-\tilde \omega_\v\big) \frac{1}{k}T_k(u(t,\v)-\tilde u(\v))\, \xi(t)\,dt\\ \\
    \displaystyle \le \int_0^T \sum_{ \v\in \V(\G)  } \vert   \omega_\v(t)- \tilde{\omega}_\v\vert\, \xi(t)\,dt.
  \end{array}
\end{equation}

  Then, as consequence of \eqref{e1A},  \eqref{e1B}  \eqref{e1C}  and  \eqref{e1D} we get~\eqref{E2IntWeakneeds01}. 
  \end{proof}
 
 For $\beta:\mathbb{R}\to \mathbb{R}$ increasing,
$$j_\beta(r) := \int_0^r \beta(s) \, ds$$  defines a convex lower semi-continuous function such that $$\beta = \partial j_\gamma,$$ the subdifferential of $j_\beta$. Let $j^*_\beta$ be the Legendre transform of $j_\beta$, then $$\beta^{-1} = \partial j^*_\beta.$$
  
In the next result we use the following notation: for $u \in L^{1}(\G)$, $j^*_{\overline{\gamma}}(u)$ is defined as 
 $$[j^*_{\overline{\gamma}}(u)]_\e = j^*_{\gamma_\e}([u]_\e) \quad \hbox{for all} \ \e \in E(\G). $$

\begin{theorem}\label{ExtUniqWea} Let  $\G$ be a  connected and compact metric graph. For each $\e\in \E(\G)$, let $1 < p_\e < \infty$ be and   $\gamma_\e:\R\to \R$ be a continuous and increasing function with $\gamma_\e(\mathbb{R})=\mathbb{R}$ and $\gamma_\e(0)=0$.   Let $f\in L^{\overline{p'}}_{loc}([0,+\infty);L^{\overline{p'}}(\G))$  and   $\omega\in L^{\infty}_{loc}([0,+\infty);L^{1}(\V(\G)))$  be.
  For every initial data 
$v_0 \in L^{1}(\G)$ 
  with $\displaystyle \int_\G j^*_{\overline{\gamma}}(v_{0})<+\infty$, there exists a unique weak solution of Problem~\eqref{evpbintro1711}. 
  
 Moreover, if, for  $i=1,2$,  $f_i  \in L^{\overline{p'}}_{loc}([0,+\infty);L^{\overline{p'}}(\G))$, $\omega_i\in L^{\overline{p'}}_{loc}([0,+\infty);L^{1}(\V(\G)))$,  and $v_i$
  is weak solution for initial data $v_{i,0}\in L^{1}(\G)$  
  with $\displaystyle \int_\G j^*_{\overline{\gamma}}(v_{i,0})<+\infty$, then, for any $t>0$:
$$
\begin{array}{c}
  \displaystyle\int_\G(v_1(t) - v_2(t))^+  \leq \int_\G (v_{1,0}  - v_{2,0})^+
    + \int_0^t \int_\G(f_1(s) - f_2(s))^+\,dxds
   \\ \\
  \displaystyle 
    +\sum_{\v\in\V(\G)}\int_0^t(\omega_1(s,\v)-\omega_2(s,\v))^+\,ds.
\end{array}    
$$
\end{theorem}

\begin{proof} 
  By Theorems \ref{EUm-accretive} and~\ref{Dominio}, we only need to prove that the mild solution obtained in Theorem \ref{Existencemild}  is a weak solution  of Problem~\eqref{evpbintro1711},  the rest is consequence of the properties of   mild solutions and Theorem~\ref{uniqweak001} (see Theorems~\ref{EUm-accretive} and~\ref{ContPrinTT}).

 Let $T >0$ and $W = (v,0)$ be the mild solution in $[0,T]$, with $v(0) = v_0$, obtained in  Theorem \ref{Existencemild}. Our aim is to prove that $v$ is a weak solution  of problem \eqref{evpbintro1711}. 
 
   For $n \in \N$,  consider a subdivision $t_0^n = 0 < t_1^n < \cdots < t_{n-1}^n < t_n^n = T$ with $t_i^n - t_{i-1}^n = \frac{T}{n}$. Given $\epsilon >0$ there exists $n\in \mathbb{N}$ such that $\frac{T}{n}\le\epsilon$, and there exist $f_1^n, \ldots,f_n^n \in L^{\overline{p'}}(\G)$, $\omega_1^n, \ldots, \omega_n^n \in L^{1}(\V(\G))$  such that
\begin{equation}\label{dataprox001}
\sum_{i=1}^n \int_{t_{i-1}^n}^{t_i^n} \left( \sum_{\e \in\E(\G)}\int_{0}^{\ell_\e} \vert [f(t)]_\e - [f_i^n]_\e \vert^{p_\e'}  \right)dt\leq \epsilon,   
\end{equation}
\begin{equation}\label{dataprox001foromega}
\sum_{i=1}^n \operatorname*{ess\,sup}_{t\in[t_{i-1}^n,t_i^n]} \Vert \omega(t) - \omega_i^n \Vert_{L^{1}(\V(\G))} \leq \epsilon,
\end{equation}
and
\begin{equation}\label{maildWeak1}
v(t) = \lim_{n \to +\infty} v_n(t) \quad \hbox{in $L^1(\G)$, uniformly for }  t \in [0,T],
\end{equation}
where  $v_n(t)$ is given  by
\begin{equation}\label{maildWeak2}
\begin{array}{l}
 v_n(0) = v_0, \\ \\
v_n(t) = v_i^n  \quad \hbox{for} \ t \in ]t_{i-1}^n, t_1^n], \quad i= 1, \ldots, n,
\end{array}
\end{equation}
with $(v_i^n,0)$ being {\it the solution} of
$$\frac{(v_i^n,0) - (v_{i-1}^n,0)}{T/n} + \mathcal{A}(v_i^n,0)\ni (f_i^n,\omega_i^n).$$
 
For such solution $(v_i^n,0)$, there exist $u_i^n  \in \mathcal{W}^{1,\overline{p}}(\G))$ with $v_i^n  = \overline{\gamma} (u_i^n)$, $i= 1, \ldots, n$, satisfying  
\begin{equation}\label{maildWeak3} 
\sum_{\e \in\E(\G)}\int_{0}^{\ell_\e}|[u_i^n]_\e'|^{p_\e-2}[u_i^n]_\e' [\varphi]_\e' \, d x  + \int_\G  \frac{v_i^n - v_{i-1}^n}{T/n} \varphi\, d x  = \int_{\G} f_i^n \varphi\,dx+ \sum_{\v\in \V(\G)}\omega_i^n(\v)\varphi(\v), 
\end{equation}
 for all   $\varphi \in \mathcal{W}^{1,\overline{p}}(\G)$.
Taking $\varphi = u_i^n$  
we get 
\begin{equation}\label{maildWeak3N} 
\sum_{\e \in\E(\G)}\int_{0}^{\ell_\e}|[u_i^n]_\e'|^{p_\e}   + \int_\G  \frac{v_i^n - v_{i-1}^n}{T/n} u_i^n\,  d x  = \int_{\G} f_i^n u_i^n\, dx + \sum_{\v\in \V(\G)}\omega_i^n(\v) u_i^n(\v). 
\end{equation}

Since  $[u_i^n]_\e = \gamma_\e^{-1}([v_i^n]_\e) = \partial j^*_{\gamma_\e}([v_i^n]_\e)$,
$$j^*_{\gamma_\e}([v_{i-1}^n]_\e) - j^*_{\gamma_\e}([v_{i}^n]_\e) \geq ([v_{i-1}^n]_\e -[v_{i}^n]_\e) [u_i^n]_\e,$$
we have
$$\int_\G  \frac{j^*_{\overline{\gamma}}(v_{i}^n) - j^*_{\overline{\gamma}}(v_{i-1}^n)}{T/n}    \leq  \int_\G  \frac{v_i^n - v_{i-1}^n}{T/n} u_i^n\, d x.$$
Therefore,   by \eqref{maildWeak3N}, we get
$$\sum_{\e \in\E(\G)}\int_{0}^{\ell_\e}|[u_i^n]_\e'|^{p_\e}   +  \int_\G  \frac{j^*_{\overline{\gamma}}(v_{i}^n) - j^*_{\overline{\gamma}}(v_{i-1}^n)}{T/n}     \leq  \int_{\G} f_i^n u_i^n\,dx+ \sum_{\v\in \V(\G)}\omega_i^n(\v) u_i^n(\v). $$
 Then, integrating in time, we obtain
 \begin{equation}\label{maildWeak4} \begin{array}{ll}
\displaystyle\sum_{i=1}^n \displaystyle \int_{t_{i-1}^n}^{t_i^n} \sum_{\e \in\E(\G)}\int_{0}^{\ell_\e}|[u_i^n]_\e'|^{p_\e}    + \int_\G  (j^*_{\overline{\gamma}}(v_{i}^n) - j^*_{\overline{\gamma}}(v_{i-1}^n))    \\
\\ \displaystyle
\leq  \sum_{i=1}^n \displaystyle \int_{t_{i-1}^n}^{t_i^n} \left( \int_{\G} f_i^n u_i^n \,dx+ \sum_{\v\in \V(\G)}\omega_i^n(\v) u_i^n(\v) \right)\end{array}
 \end{equation}
 Consequently, if we set
$$ u_n(t) = u_i^n, \ f_n(t) =f_i^n, \  \omega_n(t) = \omega_i^n \quad \hbox{for} \ t \in ]t_{i-1}, t_1], \quad i= 1, \ldots, n,$$
it follows that  
   \begin{equation}\label{maildWeak4NN}\begin{array}{ll}
\displaystyle \int_0^T \left(\sum_{\e \in\E(\G)}\int_{0}^{\ell_\e}| u_n(t)'|^{p_\e}   \right)dt  + \int_\G  (j^*_{\overline{\gamma}}(v_{n}^n) - j^*_{\overline{\gamma}}(v_{0}))  \, d x  \\ \\ \leq \displaystyle  \int_0^T  \int_{\G} f_n (t) u_n(t) + \int_0^T \sum_{\v\in \V(\G)}  \omega_n(t,\v) u_n(t,\v). \end{array}
 \end{equation}

From~\eqref{maildWeak4NN}, using that $j^*_{\gamma_\e}\ge 0$, we have 
\begin{equation}\label{1245261}
\begin{array}{c}
\displaystyle
\int_0^T \left( \sum_{\e \in\E(\G)}\int_{0}^{\ell_\e}|[u_n(t)]_\e'|^{p_\e}  \right)  dt \\
\\
\displaystyle\leq \int_\G j^*_{\overline{\gamma}}(v_{0})   + \displaystyle  \int_0^T  \int_{\G} f_n (t) u_n(t) \,dxdt + \int_0^T \sum_{\v\in \V(\G)}  \omega_n(t,\v) u_n(t,\v)\, dt.
\\
\\
\displaystyle\leq \int_\G j^*_{\overline{\gamma}}(v_{0}) + \int_0^T   \left(\sum_{\e \in\E(\G)}\left(\int_{0}^{\ell_\e}\vert [f_n(t)]_\e\vert^{p_\e'}\right)^{1/{p_\e'}}  \left(\int_{0}^{\ell_\e}\vert [u_n(t)]_\e\right)^{1/{p_\e}} \right)  dt 
\\
\\
\displaystyle   
 + \operatorname*{ess\,sup}_{t\in[0,T]}
   \Vert \omega_n(t)\Vert_{L^\infty(\V(\G))}   \int_0^T  \Vert u_n(t)  \Vert_{L^1(\V(\G))}\,  dt  .
\end{array}
\end{equation}

 By \eqref{maildWeak1}, there exists a constant $C_1 >0$ such that
\begin{equation}\label{maildWeak61}
\sup_{t \in [0,T]}\int_\G \vert v_n(t) \vert \leq C_1\quad \forall n\in \N.
\end{equation} 
Then, applying Lemma \ref{Poinc} in each edge, we have that there exists $C_2>0$ (not depending on $\e$, since there is a finite quantity of edges) such that
\begin{equation}\label{maildWeak62}
\left(\int_{0}^{\ell_\e}\vert [u_n(t)]_\e\vert^{p_\e} \right)^{1/{p_\e}} \leq C_2\left( \left( \int_{0}^{\ell_\e}\vert [u_n(t)]_\e'\vert^{p_\e}\right)^{1/{p_\e}} +1 \right),
\end{equation}
  for all $n \in \N$, $t \in [0,T]$ and $\e \in E(\G)$. 
On the other hand, by the Trace Theorem (applied to each segment, and adding the corresponding inequalities), there is a constant $C_3 >0$     such that
\begin{equation}
 \Vert u_n(t)  \Vert_{L^1(\V(\G))}  \leq C_3\left(\sum_{\e \in\E(\G)}  \left(\int_{0}^{\ell_\e}|[u_n(t)]_\e|^{p_\e}\right)^{1/p_\e}  + \sum_{\e \in\E(\G)}  \left( \int_{0}^{\ell_\e}|[u_n(t)]_\e'|^{p_\e} \right)^{1/p_\e}   \right),
\end{equation}
  for all $n \in \N$, $t \in [0,T]$. 
And, on account of~\eqref{maildWeak62},  there exists a constant $C_4>0$ such that
\begin{equation}\label{maildWeak63}
 \Vert u_n(t)  \Vert_{L^1(\V(\G))}  \leq C_4\left(   \sum_{\e \in\E(\G)}  \left( \int_{0}^{\ell_\e}|[u_n(t)]_\e'|^{p_\e} \right)^{1/p_\e} +1  \right) \quad \forall n \in \N, \, t \in [0,T]. 
\end{equation}

From~\eqref{1245261}, \eqref{maildWeak62} and~\eqref{maildWeak63}, we have that  
\begin{equation}\label{124526102}
\begin{array}{c}
\displaystyle
\int_0^T \left( \sum_{\e \in\E(\G)}\int_{0}^{\ell_\e}|[u_n(t)]_\e'|^{p_\e}  \right)  dt \le \int_\G j^*_{\overline{\gamma}}(v_{0})
\\ \\
\displaystyle   + C_2\int_0^T   \left(\sum_{\e \in\E(\G)}\left(\int_{0}^{\ell_\e}\vert [f_n(t)]_\e\vert^{p_\e'}\right)^{1/{p_\e'}}   \left( \int_{0}^{\ell_\e}\vert [u_n(t)]_\e'\vert^{p_\e}\right)^{1/{p_\e}} +1  \right)  dt 
\\
\displaystyle  
 + C_4 \operatorname*{ess\,sup}_{t\in[0,T]}
   \Vert \omega_n(t)\Vert_{L^\infty(\V(\G))}   \int_0^T   \left(   \sum_{\e \in\E(\G)}  \left( \int_{0}^{\ell_\e}|[u_n(t)]_\e'|^{p_\e} \right)^{1/p_\e} +1  \right)   dt  .
\end{array}
\end{equation}
 Therefore, using Young's inequality, and taking into  account theintegrability of the data, there exists a constant $C_5>0$ such  that
\begin{equation}\label{Young1}
\displaystyle \int_0^T \left( \sum_{\e \in\E(\G)}\int_{0}^{\ell_\e}|[u_n(t)]_\e'|^{p_\e}\right)dt \le C_5  \quad \forall \, n \in \N. 
\end{equation}

 Consequently, from \eqref{Young1} and \eqref{maildWeak62}, 
  a constant $C_6 >0$ such that
 \begin{equation}\label{maildWeakNM}
  \int_0^T \left(\sum_{\e \in\E(\G)} \int_{0}^{\ell_\e}\Big(|[u_n(t)]_\e|^{p_\e} +|[u_n(t)]_\e'|^{p_\e}\Big) \right)dt   \leq C_6  \quad\forall n\in \N.
\end{equation}
Thus, by this estimation , there exists a subsequence, denoted again $\{ u_n \}$, such that  
\begin{equation}\label{maildWeak7}
u_n \to u \quad \hbox{weakly in }    L^{\overline{p}}(0, T; \mathcal{W}^{1,\overline{p}}(\G)) \quad \hbox{as }   n \to +\infty, 
\end{equation}

By \eqref{maildWeak1},
 we have
$$v_n \to v \in L^1((0,T) \times \G).$$
Then,  since $v_n = \overline{\gamma}(u_n)$, having in mind   the above convergence and   \eqref{maildWeak7},  by~\cite[Lemma~G]{BCS88}, we get that $$v = \overline{\gamma}(u).$$

We denote $\a_n$ the function in $(0,T) \times \G$, such that
$$[\a_n(t)]_e = \vert [u_n(t)^{\prime}]_\e \vert^{p_\e-2}  [u_n(t)^{\prime}]_\e.$$
Since $\{ u'_n \}$ is bounded in $L^{\overline {p}}((0,T) \times \G)$, it follows  that $\{ \a_n \}$ is bounded in $L^{\overline {p^{\prime}}}((0,T) \times \G)$, so we can assume that
$$\a_n\to \Phi \quad \hbox{weakly in} \ \ L^{\overline {p^{\prime}}}((0,T) \times \G)) \ \hbox{as} \ n \to +\infty.$$
From \eqref{maildWeak3}, given  $\psi \in W^{1,1}(0,T; \mathcal{W}^{1,1}(\G)) \cap L^p(0,T; \mathcal{W}^{1,\overline{p}}(\G))$ with $\psi(0) = \psi (T) =0$, we have  (define $v_n(t)=v_0$ for $t<0$):
\begin{equation}\label{maildWeak9}
\begin{array}{c}
\displaystyle \int_0^T  \sum_{\e \in\E(\G)}\int_{0}^{\ell_\e}|[u_n(t)]_\e'|^{p_\e-2}[u_n(t)]_\e'  [\psi(t)]_\e'\, d x dt+ \int_\G \int_0^T \frac{v_n(t) - v_n(t- T/n)}{T/n} \psi(t) \, dxdt\\ \\=  \displaystyle\int_0^T \int_{\G} f_n (t) \psi(t)\, dxdt+  \int_0^T \sum_{\v\in \V(\G)}\omega_n(t,\v)\psi(t,\v)dt.
\end{array}
\end{equation}
Now
$$\begin{array}{c}
\displaystyle\lim_{n \to +\infty} \int_\G \int_0^T \frac{v_n(t) - v_n(t- T/n) } {T/n} \psi(t)\, dxdt \\ \\
\displaystyle= \lim_{n \to +\infty} \Big(-  \int_\G  \int_0^{T-T/n} v_n (t) \frac{\psi(t+T/n) - \psi(t)}{T/n}\,dxdt\qquad\qquad
\\ \\
\displaystyle \qquad\qquad
+  \int_\G  \int_{T-T/n}^T \frac{v_n (t) \psi(t)}{T/n}\, dxdt - \int_\G  \int_{0}^{T/n} \frac{v_0 \psi(t)}{T/n}\, dxdt\Big)
\\ \\
\displaystyle= - \int_0^T \int_\G v(t) \frac{\partial \psi}{\partial t}\, dxdt. 
\end{array}
$$
Therefore, taking limits in \eqref{maildWeak9} as $n \to +\infty$, we obtain
\begin{equation}\label{maildWeak10}
\int_0^T \int_\G \Phi(t) \psi'(t) - \int_0^T \int_\G v(t) \frac{\partial \psi}{\partial t} = \displaystyle\int_0^T \int_{\G} f(t) \psi(t)+  \int_0^T \sum_{\v\in \V(\G)}\omega(t,\v)\psi(t,\v)
\end{equation}

Thus,  to finish   the proof we only need show that
\begin{equation}\label{maildWeakF}[\Phi(t)]_\e =  |[u(t)]_\e'|^{p_\e-2}[u(t)]_\e' [\psi(t]_\e'\quad \forall \, \e \in\E(\G) .\end{equation}
To do that,  let us first prove the following inequality,
\begin{equation}\label{maildWeak11}
\limsup_{n\to+\infty} \int_0^T \sum_{\e \in\E(\G)}\int_{0}^{\ell_\e}| [u_n(t)]_\e'|^{p_\e} \, \d x dt \leq \int_0^T \int_\G \Phi(t) u(t)'.
\end{equation}
 Thanks to \eqref{maildWeak4NN} and Fatou's Lemma, we have
  \begin{equation}\label{maildWeak12}\begin{array}{c}
\displaystyle\limsup_{n\to+\infty} \displaystyle\int_0^T \sum_{\e \in\E(\G)}\int_{0}^{\ell_\e}| [u_n(t)]_\e'|^{p_\e} \, \d x dt \\ \\ \leq -  \displaystyle \int_\G  (j^*_{\overline{\gamma}}(v(T)) - j^*_{\overline{\gamma}}(v_{0}))  \, d x +\displaystyle  \int_0^T  \int_{\G} f (t) u(t) + \int_0^T \sum_{\v\in \V(\G)}  \omega(t,\v) u(t,\v). \end{array}
 \end{equation}
 On the other hand \eqref{maildWeak10} can be rewritten as
 \begin{equation}\label{maildWeak13}
 \int_0^T \int_\G v(t) \frac{\partial \psi}{\partial t} = \int_0^T (F(t),\psi(t)),
\end{equation}
 where $F$ is given by
$$(F(t),\psi(t)) =  \int_\G \Phi(t) \psi'(t) - \displaystyle    \int_{\G} f (t) \psi(t) - \sum_{\v\in \V(\G)}  \omega(t,\v) \psi(t,\v) $$
Let us see now that 
  \begin{equation}\label{maildWeak14}
- \frac{d}{dt} \int_G j^*_{\overline{\gamma}}(v(t)) = (F,u) \quad \hbox{in}  \ \ \mathcal{D}'(]0,T[).
  \end{equation}
  Indeed, given $0 \leq \varphi \in \mathcal{D}(]0,T[$ and $\tau >0$ small, we define
  $$\eta_\tau(t):= \frac{1}{\tau} \int_t^{t + \tau}  u(s) \varphi(s) ds.$$
  Taking $\eta_\tau$ as test function in \eqref{maildWeak13}, we get
$$\int_0^T (F(t),\eta_\tau(t)) =  \int_0^T \int_\G v(t) \frac{\partial \eta_\tau}{\partial t} = \int_0^T \int_\G v(t) \varphi^{\prime}(t)$$ $$= \int_0^T \int_\G v(t) \frac{u(t + \tau) \varphi(t + \tau)- u(t) \varphi(t)}{\tau}  = \int_0^T \int_\G \frac{v(t-\tau) - v(t)}{\tau} u(t) \varphi(t).$$
Now, 
$$u(t) = \overline{\gamma }^{-1}(v(t)) = \partial j^*_{\overline{\gamma}}(v(t)),$$
hence
$$u(t)(v(t-\tau) - v(t)) \leq j^*_{\overline{\gamma}}(v(t-\tau)) - j^*_{\overline{\gamma}}(v(t)).$$
Thus, we obtain that
$$\int_0^T (F(t),\eta_\tau(t)) \leq  \frac{1}{\tau}\int_0^T j^*_{\overline{\gamma}}(v(t-\tau)) - j^*_{\overline{\gamma}}(v(t)) \varphi(t)  =  \int_0^T j^*_{\overline{\gamma}}(v(t)) \frac{\varphi(t+\tau) - \varphi(t)}{\tau}.$$
Then, taking limits as $\tau \to 0^+$, we get
$$\int_0^T (F(t),u(t)) \leq  \int_0^T j^*_{\overline{\gamma}}(v(t)) \varphi^{\prime}(t).$$
Taking now
$$\overline{\eta}_\tau(t):= \frac{1}{\tau} \int_t^{t + \tau} u(s - \tau) \varphi(s) ds,$$
 and arguing  as above we get the another inequality, and we conclude the proof of \eqref{maildWeak14}.
  
Integrating from $0$ to $T$ in \eqref{maildWeak14}, we have 
 $$ \int_0^T\int_\G \Phi(t) u'(t) =-\int_\G  (j^*_{\overline{\gamma}}(v(T)) - j^*_{\overline{\gamma}}(v_{0}))  \, d x  +  \displaystyle  \int_0^T  \int_{\G} f (t) \psi(t) + \int_0^T \sum_{\v\in \V(\G)}  \omega(t,\v) \psi(t,\v).$$
 Hence, using \eqref{maildWeak10} we obtain \eqref{maildWeak11}.

 Finally, to prove \eqref{maildWeakF} we apply the Minty-Browder's method. Indeed, for $\rho \in  L^{\overline{p}}(0,T; \mathcal{W}^{1,\overline{p}}(\G))$, we have  
 $$
 \begin{array}{c}
 \displaystyle
 \int_0^T  \sum_{\e \in\E(\G)}\int_{0}^{\ell_\e} [\rho(t)]_\e' \vert^{p_\e -2} [\rho(t)]_\e' ([u_n(t)]_\e' - [\rho(t)]_\e') \,dxdt
 \\ \\ \displaystyle
 \leq  \int_0^T \sum_{\e \in\E(\G)}\int_{0}^{\ell_\e} \vert [u_n(t)]_\e' \vert^{p_\e -2} [u_n(t)]_\e' ([u_n(t)]_\e' - [\rho(t)]_\e')\, dxdt,
 \end{array}
 $$
 so that, passing to the limit and using \eqref{maildWeak11}, we get  
 $$ \int_0^T  \sum_{\e \in\E(\G)}\int_{0}^{\ell_\e} [\rho(t)]_\e' \vert^{p_\e -2} [\rho(t)]_\e' ([u(t)]_\e' - [\rho(t)]_\e') \,dxdt  \\ \\
 \leq \int_0^T \int_\G \Phi(t)(u(t)' - \rho(t)')\, dxdt.
 $$
 Then, taking $\rho =  u \pm \lambda \xi$ for $\lambda >0$ and $\xi \in  L^p(0,T; \mathcal{W}^{1,\overline{p}}(\G))$, we get  
 $$  
 \int_0^T  \sum_{\e \in\E(\G)}\int_{0}^{\ell_\e}\vert ([u(t)]_\e+\lambda [\xi(t)]_\e)' \vert^{p_\e -2}([u(t)]_\e+ \lambda [\xi(t)]_\e)' [\xi(t)]_\e'\, dxdt \ge  \int_0^T \int_\G \Phi(t) \xi(t)'\,dxdt, $$
 and
 $$  
 \int_0^T  \sum_{\e \in\E(\G)}\int_{0}^{\ell_\e}\vert ([u(t)]_\e-\lambda [\xi(t)]_\e)' \vert^{p_\e -2}([u(t)]_\e- \lambda [\xi(t)]_\e)' [\xi(t)]_\e'\, dxdt \le  \int_0^T \int_\G \Phi(t) \xi(t)'\,dxdt, $$
 so that, letting $\lambda \to 0$, we obtain  
  $$  \int_0^T  \sum_{\e \in\E(\G)}\int_{0}^{\ell_\e} \vert [u(t)]_\e'\vert^{p_\e -2}[u(t)]_\e' [\xi(t)]_\e'\,dxdt= \int_0^T \int_\G \Phi(t) \xi(t)' \,dxdt$$
  for all $\xi \in  L^p(0,T; \mathcal{W}^{1,\overline{p}}(\G)),$
 which implies that \eqref{maildWeakF} holds.
     \end{proof}

{\flushleft \bf Acknowledgements.} The authors would like to thank Delio Mugnolo for some discussion on this issue.   The  authors  have been partially supported  by  Grant PID2022-136589NB-I00 funded by MCIN/AEI/10.13039/501100011033 and FEDER and by Grant RED2022-134784-T funded by MCIN/AEI/10.13039/501100011033.

For the purpose of open access, the authors have applied a CC BY public copyright license to any Author Accepted Manuscript version arising from this submission.

{\bf Data Availability.} Data sharing is not applicable to this article as no datasets were generated or
analyzed during the current study.

{\bf Conflict of interest.} The authors have no conflict of interest to declare that are relevant to the content of this
article.

\newcommand{\etalchar}[1]{$^{#1}$}

\end{document}